\documentclass[a4paper,twoside,reqno]{amsart}
\newcommand{\Ueberschrift}{Infinite series of quaternionic $1$-vertex cube complexes,\\[1ex]  the doubling construction, and explicit cubical Ramanujan complexes}
\newcommand{\Kurztitel}{Quaternionic $1$-vertex cube complexes, non-residually finite lattices and cubical Ramanujan complexes}
\usepackage[headings]{fullpage}
\usepackage{amsmath} 
\usepackage{amssymb} 
\usepackage{amsopn}
\usepackage{amsthm}
\usepackage{amsfonts} 
\usepackage{mathrsfs}
\usepackage{mathtools}
\usepackage{stmaryrd}
\usepackage{manfnt} \usepackage{mathdots}
\usepackage[OT2, T1]{fontenc}
\usepackage[dvips]{graphicx} \usepackage[usenames]{color}
\usepackage[all]{xy}
\usepackage{dsfont}
\usepackage{ifthen}
\usepackage{marvosym}
\usepackage{wasysym}
\usepackage{marginnote}
\usepackage{float}
\usepackage{float}
\usepackage{booktabs}
\usepackage{subfigure}
\usepackage{tikz}
\usepackage{tkz-euclide}
\usetikzlibrary{through,intersections,calc,angles}
\usetikzlibrary{shapes.geometric}
\usetikzlibrary{decorations.markings,calc}
\usepackage{pgfplots}
\pgfplotsset{compat=1.12}
\usepackage[pdfpagelabels, pdftex]{hyperref}
\hypersetup{
  pdftitle={},
  pdfauthor={},
  pdfsubject={},
  pdfkeywords={},
  colorlinks=true,    
  linkcolor=red,     
  citecolor=blue,     
  filecolor=blue,      
  urlcolor=blue,       
  breaklinks=true,
  bookmarksopen=true,
  bookmarksnumbered=true,
  pdfpagemode=UseOutlines,
  plainpages=false}
\DeclareRobustCommand{\gobblefive}[5]{}

\DeclareMathOperator{\rH}{H}

\DeclareMathOperator{\rM}{M}

\DeclareMathOperator{\rS}{S}
\DeclareMathOperator{\rT}{T}

\newcommand{\bA}{{\mathbb A}}

\newcommand{\bC}{{\mathbb C}}

\newcommand{\bF}{{\mathbb F}}
\newcommand{\bG}{{\mathbb G}}

\newcommand{\bN}{{\mathbb N}}

\newcommand{\bP}{{\mathbb P}}
\newcommand{\bQ}{{\mathbb Q}}

\newcommand{\bT}{{\mathbb T}}

\newcommand{\bZ}{{\mathbb Z}}

\newcommand{\cA}{{\mathscr A}}

\newcommand{\cC}{{\mathscr C}}
\newcommand{\cD}{{\mathscr D}}

\newcommand{\cT}{{\mathscr T}}

\newcommand{\cX}{{\mathscr X}}

\newcommand{\dA}{{\mathcal A}}

\newcommand{\dH}{{\mathcal H}}

\newcommand{\dO}{{\mathcal O}}

\newcommand{\dT}{{\mathcal T}}

\newcommand{\fO}{{\mathfrak O}}

\newcommand{\fg}{{\mathfrak g}}

\newcommand{\fm}{{\mathfrak m}}

\newcommand{\fo}{{\mathfrak o}}
\newcommand{\fp}{{\mathfrak p}}

\DeclareSymbolFont{cyrletters}{OT2}{wncyr}{m}{n}
\DeclareMathSymbol{\Sha}{\mathalpha}{cyrletters}{"58}

\DeclareMathOperator{\Aut}{Aut}

\newcommand{\ev}{{\rm ev}}

\DeclareMathOperator{\Hom}{Hom}

\DeclareMathOperator{\im}{im}
\newcommand{\inj}{\hookrightarrow}

\newcommand{\one}{\mathbf{1}}

\DeclareMathOperator{\pr}{pr}

\newcommand{\sets}{{\sf sets}}

\newcommand{\surj}{\twoheadrightarrow}

\newcommand{\xyinj}{\ar@{^(->}}

\DeclareMathOperator{\GL}{GL}

\DeclareMathOperator{\PGL}{PGL}
\DeclareMathOperator{\PSL}{PSL}

\DeclareMathOperator{\SL}{SL}

\newcommand{\matzz}[4]{\left(
\begin{array}{cc} #1 & #2 \\ #3 & #4 \end{array} \right)}

\newcommand{\Gm}{\bG_{\rm m}}

\DeclareMathOperator{\Pic}{Pic}

\DeclareMathOperator{\Spec}{Spec}

\DeclareMathOperator{\disc}{disc}

\DeclareMathOperator{\ord}{ord}

\DeclareMathOperator{\res}{res}

\newcommand{\BTT}{\dT} 
\newcommand{\RT}{\rT} 
\newcommand{\PT}{\mathbf{T}} 
\def\10{{\overrightarrow{10}}}
\def\01{{\overrightarrow{01}}}

\DeclareMathOperator{\Cube}{{\sf Cube}}

\newcommand{\ep}{\varepsilon}

\DeclareMathOperator{\Nrd}{{\rm Nrd}}
\newcommand{\op}{{\rm op}}
\newcommand{\ph}{\varphi}

\newcommand{\ov}[1]{\mbox{${\overline{#1}}$}} 
\DeclarePairedDelimiter\abs{\lvert}{\rvert}

\newtheorem{thm}{Theorem}[section]

\newtheorem{prop}[thm]{Proposition}
\newtheorem{lem}[thm]{Lemma}

\newtheorem{cor}[thm]{Corollary}
\newtheorem{conj}[thm]{Conjecture}

\newtheorem{exABC}{Example}

\theoremstyle{definition}
\newtheorem{defi}[thm]{Definition}

\theoremstyle{remark}
\newtheorem{rmk}[thm]{Remark}

\newtheorem{nota}[thm]{Notation}
\newtheorem{ex}[thm]{Example}

\newenvironment{pro*}[1][\proofname]{{\it{#1:}} }{}
\newenvironment{pro**}[1][]{{\it{#1}} }{\hfill $\square$}

\numberwithin{equation}{section}
\begin{document}

\hrule width\hsize

\vskip 0.9cm

\title[\Kurztitel]{\Large \Ueberschrift} 
\author{Nithi Rungtanapirom}
\address{Nithi Rungtanapirom, Institut f\"ur Mathematik, Goethe--Universit\"at Frankfurt, Ro\-bert-Mayer-Stra{\ss}e {6--8},
60325~Frankfurt am Main, Germany}
\email{rungtana@math.uni-frankfurt.de}

\author{Jakob Stix}
\address{Jakob Stix, Institut f\"ur Mathematik, Goethe--Universit\"at Frankfurt, Ro\-bert-Mayer-Stra{\ss}e~{6--8},
60325 Frankfurt am Main, Germany}
\email{stix@math.uni-frankfurt.de}

\author{Alina Vdovina}
\address{Alina Vdovina, School of Mathematics and Statistics, Newcastle University, Newcastle upon Tyne, NE1 7RU, UK}
\email{alina.vdovina@ncl.ac.uk}

\date{August 9, 2018}

\maketitle

\begin{quotation} 
\noindent \small {\bf Abstract} --- We construct vertex transitive lattices on products of trees of arbitrary dimension $d \geq 1$ based on quaternion algebras over global fields with exactly two ramified places. Starting from arithmetic examples, we find non-residually finite groups generalizing earlier results of Wise, Burger and Mozes to higher dimension. We make effective use of the combinatorial language of cubical sets and the doubling construction generalized to arbitrary dimension.  

Congruence subgroups of these quaternion lattices yield explicit cubical Ramanujan complexes, a higher dimensional cubical version of Ramanujan graphs (optimal expanders).
\end{quotation}

\setcounter{tocdepth}{1} {\scriptsize \tableofcontents}
\section{Introduction}
Bass--Serre theory is concerned with discrete groups acting on trees. The theory of irreducible lattices acting on products of trees  provides a higher-dimensional generalization of Bass--Serre theory.  
It is a remarkable discovery, due to Burger and Mozes \cite{burger-mozes:groupsontrees,burger-mozes:lattices}  \cite{mozes:survey}, that lattices in products of trees behave similar to irreducible lattices in higher rank semisimple non-archimedean Lie  groups: they have a rich structure with superrigidity, normal subgroup theorems and arithmeticity theorems \`a la Margulis. For the most recent developments on the theory of groups acting on trees see \cite{caprace:survey}.

The methods developed for the study  of lattices in $\Aut(\rT)$,  for a single tree $\rT$, were first extended to the study of irreducible lattices in a product of two trees. 
Unlike irreducible higher rank lattices in semisimple non-archimedean Lie  groups, these lattices in products of two trees may be non-residually finite. The first non-residually finite examples acting on products of two trees were obtained by Wise \cite{wise:thesis}. The results of Burger, Moses, Zimmer  \cite{burger-mozes:simple,burger-mozes:lattices, BMZ} on non-residually finite and simple cocompact lattices in  the group of automorphisms of a product of two trees  are described in \cite{mozes:survey}. Further examples of non-residually finite and simple lattices acting on products of two trees were obtained in \cite{rattaggi:thesis, rattaggi:simple}.  
 
\smallskip
 
In this paper, we construct infinite series of explicit examples of irreducible $S$-arithmetic quaternionic lattices acting simply transitively on the vertices of products of $n\geq 1$ trees of constant valency, see 
Theorem~\ref{thm:structure_Gamma_S} for quaternions over $\bF_q(t)$ 
and Theorem~\ref{thm:torsionfreelattice_S} for quaternions over $\bQ$, in both cases exactly $2$ places are ramified. 
These series of lattices extend the examples for $n=2$ from \cite{stixvdovina:fakequadric3} in the function field case, and those from \cite[\S2.4]{burger-mozes:lattices} \cite{rattaggi:thesis} in the case of Hurwitz quaternions over $\bQ$. 
The emphasis here is on the range $n \geq 3$ and on the word `explicit', at least for the infinite series of examples over global function fields, because our finite presentations of these lattices translate immediately into a combinatorial description of the quotient cube complex. Here explicit means that one has to solve a finite list of discrete logarithms in the multiplicative group of a finite field and from the exponents so obtained can write down a list of $4$-term relations among the generators. For details we refer to Section~\S\ref{sec:concretemodel}.  Here are the smallest examples in each case:

\begin{exABC}[Example \ref{ex:666}]
The following is an irreducible torsion-free lattice 
\[
\Gamma_{\{2,3,4\}} = \left\langle \begin{array}{c}
a_1,a_5,a_9,a_{13},a_{17},a_{21}, \\
b_2,b_6,b_{10},b_{14},b_{18},b_{22}, \\
c_3,c_7,c_{11},c_{15},c_{19},c_{23}
\end{array}
\ \left| 
\begin{array}{c}
a_ia_{i+12} = b_i b_{i+12} = c_ic_{i+12} = 1  \ \text{ for all $i$ }, \\

a_{1}b_{2}a_{17}b_{22}, \ 
a_{1}b_{6}a_{9}b_{10}, \ 
a_{1}b_{10}a_{9}b_{6}, \ 
a_{1}b_{14}a_{21}b_{14}, \ 
a_{1}b_{18}a_{5}b_{18}, \\ 
a_{1}b_{22}a_{17}b_{2}, \ 
a_{5}b_{2}a_{21}b_{6}, \ 
a_{5}b_{6}a_{21}b_{2}, \ 
a_{5}b_{22}a_{9}b_{22}, \\

a_{1}c_{3}a_{17}c_{3}, \ 
a_{1}c_{7}a_{13}c_{19}, \ 
a_{1}c_{11}a_{9}c_{11}, \ 
a_{1}c_{15}a_{1}c_{23}, \ 
a_{5}c_{3}a_{5}c_{19}, \\
a_{5}c_{7}a_{21}c_{7}, \ 
a_{5}c_{11}a_{17}c_{23}, \ 
a_{9}c_{3}a_{21}c_{15}, \ 
a_{9}c_{7}a_{9}c_{23}, \\

b_{2}c_{3}b_{18}c_{23}, \ 
b_{2}c_{7}b_{10}c_{11}, \ 
b_{2}c_{11}b_{10}c_{7}, \ 
b_{2}c_{15}b_{22}c_{15}, \ 
b_{2}c_{19}b_{6}c_{19}, \\
b_{2}c_{23}b_{18}c_{3}, \ 
b_{6}c_{3}b_{22}c_{7}, \ 
b_{6}c_{7}b_{22}c_{3}, \ 
b_{6}c_{23}b_{10}c_{23}.
\end{array}
\right.\right\rangle .
\]
that acts vertex transitively on the product $\RT_{6} \times  \RT_{6} \times  \RT_{6}$.
\end{exABC}

\begin{exABC}[Example \ref{ex:468}]
Using $a_{i+2} = a_i^{-1}$ for $i=1,2$, and $b_{i+3} = b_i^{-1}$ for $i = 1, \ldots, 3$, and $c_{i+4} = c_i^{-1}$ for $i= 1, \ldots, 4$ we define a lattice
\[
\Gamma'_{\{3,5,7\}} = \left\langle
\begin{array}{c}
a_1,a_2 \\ 
b_1,b_2,b_3 \\
c_1,c_2,c_3,c_4
\end{array}
\ \left| 
\begin{array}{c}
a_1b_1a_4b_2,  \ a_1b_2a_4b_4, \  a_1b_3a_2b_1, \ 
a_1b_4a_2b_3,  \ a_1b_5a_1b_6, \ a_2b_2a_2b_6 \\

a_1c_1a_2c_8, \ a_1c_2a_4c_4, \ a_1c_3a_2c_2, \ a_1c_4a_3c_3, \\
a_1c_5a_1c_6, \ a_1c_7a_4c_1, \ a_2c_1a_4c_6, \ a_2c_4a_2c_7 \\

b_1c_1b_5c_4, \
b_1c_2b_1c_5, \
b_1c_3b_6c_1, \
b_1c_4b_3c_6, \
b_1c_6b_2c_3, \
b_1c_7b_1c_8, \\
b_2c_1b_3c_2, \
b_2c_2b_5c_5, \
b_2c_4b_5c_3, \
b_2c_7b_6c_4, \
b_3c_1b_6c_6, \
b_3c_4b_6c_3
\end{array}
\right.\right\rangle.
\]
that is irreducible torsion-free lattice and acts vertex transitively on the product $\RT_{4} \times  \RT_{6} \times  \RT_{8}$.
\end{exABC}

In the function field case over $\bF_q(t)$ we focus on sets $S$ of $\bF_q$-rational places. Hence, all tree factors are $(q+1)$-regular trees. While the theoretical results persist with sets $S$ that also contain non-$\bF_q$-rational places, the explicit nature of our results needs new ideas if it should be extended in this direction.

We would like to mention that the only explicit examples of vertex transitive irreducible lattices in products of buildings in dimensions $\geq 3$ prior to this work are the arithmetic groups acting on  products of ($1$-skeleta) of triangle complexes described in \cite[Prop 3.6]{glasner-mozes}. For a non-existence result of lattices acting on $\RT_6 \times \RT_6 \times \RT_6$ with alternating or symmetric local permutation action due to Radu see the survey 
\cite[\S4.12]{caprace:survey}.

\smallskip

Arithmetic groups are residually finite. Applying the doubling construction pioneered by Wise, Burger and Mozes to their quotient cube complexes, interesting phenomena happen: we get many examples of non-residually finite groups acting vertex transitively on cube complexes of any dimension, see Section~\S\ref{sec:double} and Theorem~\ref{thm:non residually finite doubles}. It is tempting to ask if these non-residually finite lattices are also virtually simple. We introduce in Section~\S\ref{sec:double} the language of cubical sets paralleling the notion of a simplicial sets. Cubical sets help us to deal with the doubling construction in higher dimension.

\smallskip

In Section~\ref{sec:ramanujan} we discuss cubical complexes built from Cayley graphs for quotients by congruence subgroups of the groups given by 
Theorem~\ref{thm:structure_Gamma_S}. They give rise to examples of cubical Ramanujan complexes. Our  definition of Ramanujan complexes is an analogue of the one given in \cite{LSV1} and \cite{JL}. For the connection with more recent results on higher-dimensional expanders, see \cite{LSV}, \cite{LSV1}, \cite{KKL}, \cite{gromovandco}, \cite{lubotzky:takagilecture}, \cite{lubotzky:highexpanders}. 
The proof of Theorem~\ref{thm:Ramanujan} that our complexes are cubical Ramanujan parallels the proof that the quotients of Euclidean buildings by Cartwright-Steger are simplicial Ramanujan, see \cite{KKL}, \cite{LSV1}. Its method of proof  occurred first in \cite{margulis} \cite{LPS} to construct families of Ramanujan graphs. We quickly summarize the method: Ramanujanness is related to the Hecke eigenvalues of an automorphic representation of a quaternion group. The Jacquet--Langlands correspondence transforms the representation to a cuspidal representation for $\GL_2$. The claim on the Hecke eigenvalues thus becomes the eponymous Ramanujan-Petersson conjecture, a theorem due to Drinfeld in the case of a global function field.

The distinctive feature of our cubical Ramanujan complexes for congruence subgroups of our vertex transitive lattices $\Gamma_S$ in the function field case lies in the explicit nature in which these higher dimensional versions of expander graphs can be studied.

\subsection*{Acknowledgments} 
We thank Shai Evra for explaining to us the connection of the spectrum of the adjacency operators via automorphic representation theory to the Ramanujan-Petersson conjecture for $\GL_2$. We thank Dani Wise for discussions related to his doubling construction of cubical complexes. We thank Alessandra Iozzi and Marc Burger for useful discussions and providing
us with the reference \cite{BMZ}. We thank Cornelia Drutu and Alain Valette for discussing connections of the theory of buildings and higher-dimensional expanders at the Newton Institute, Cambridge, spring 2017. The third author would like to thank University of Berkeley, Newton Institute Cambridge, Hunter College CUNY for providing excellent research environment during her stays.

\section{Lattices of quaternions over \texorpdfstring{$\bF_q(t)$}{global rational function fields}} \label{sec:lattice}
We review lattices constructed in \cite{stixvdovina:fakequadric3} with a slight generalization to an arbitrary prime power $q$, including powers of $2$. We provide simplified formulas and more importantly generalise to higher rank. 

\subsection{The quaternion algebra}
We fix a prime power $q$, denote by $\bF_q$ a finite field with $q$ elements and consider $K = \bF_q(t)$ as the function field of the projective line $\bP^1$ over $\bF_q$. The quaternion algebra $D$ over $K$ is defined as follows (by $K\{Z,F\}$ we denote the free associative $K$-algebra in variables $Z,F$):
\begin{itemize}
\item
If $q$ is odd, we fix a non-square $c \in \bF_q \setminus (\bF_q)^2$ and define
\[
D := \left(\frac{c,t}{K}\right) = K\{Z,F\}/\big(Z^2 = c, F^2 = t, FZ=-ZF\big).
\]
\item
If $q$ is a power of $2$, we fix an element $c\in\bF_q$ such that the polynomial $X^2+X+c\in\bF_q[X]$ is irreducible and define
\[
D := \left[\frac{c,t}{K}\right) = K\{Z,F\}/\big(Z^2+Z = c, F^2 = t, FZ=(Z+1)F\big).
\]
\end{itemize}

The construction of $D$ has the following variant that is independent of the parity of $q$. By assumption on $c$, the subalgebra 
\[
L = K\oplus K\cdot Z \simeq \bF_{q^2}(t)
\]
is a separable quadratic field extension of $K$ with a unique non-trivial $K$-automorphism denoted by 
\[
\sigma: L\to L, \qquad x\mapsto\overline{x}.
\]
If $q$ is odd, then $\sigma(Z) = -Z$, but if $q$ is even, then $\sigma(Z) = Z+1$.
It follows immediately by comparing the presentations as a $K$-algebra that $D$ is the cyclic algebra
\[
D = (L,\sigma,t) = L\{F\}/(F^2=t). 
\]
Here $L\{F\}$ is the $\sigma$-semi-linear polynomial ring with $Fx=\overline{x}F$ for all $x\in L$.
  
From the definition in terms of symbols as $(c,t/K)$ (resp.\ as $[c,t/K)$) it follows that $D$ ramifies at most in $B = \{0,\infty\}$. That $D$ actually ramifies in all places of $B$ is verified by computing its residue in $t=0$, and by noting that the number of ramified places must be even. The non-vanishing of the residue in $t=0$ is equivalent to the condition imposed on $c$.

\subsection{A sheaf of maximal orders}

The constant sheaf of division algebras $D \otimes \dO_{\bP^1}$ contains a subsheaf of maximal orders defined by 
\[
\dA = \dO \oplus \dO \cdot Z \oplus \dO(-\infty) \cdot  F \oplus \dO(-\infty) \cdot FZ \subseteq D \otimes \dO_{\bP^1}.
\]
Here $\dO(-\infty)$ is the ideal sheaf of the point $\infty \in \bP^1$ and, for an open $U \subseteq \bP^1$, the sections over $U$ are
\[
\dA(U) = \{u + vZ + xF + yFZ \in D \otimes \dO(U) \ ; \  u,v \in \Gamma(U,\dO), \ x,y \in \Gamma\big(U,\dO(-\infty)\big)\}.
\]
The subsheaf $\dA$ is closed under multiplication inherited from $D$ because the simple pole of $t=F^2$ in $\infty$ is compensated by the zeros at $\infty$ imposed for the coefficients $x$ and $y$. The discriminant of the reduced trace form yields a map
\[
\disc : \det(\dA \otimes \dA) \simeq \dO(-4\infty) \to \dO
\]
isomorphic to multiplication by $16c^2t^2$ for $q$ odd and $t^2$ for $q$ even. The cokernel of $\disc$ thus has length $4$ with a contribution of $2$ at $t=0$ and $t = \infty$ each. This shows by the general theory of reduced trace forms that $\dA$ is indeed a sheaf of maximal orders. 

The reduced norm defines a map of sheaves of (multiplicative) monoids 
\[
\Nrd :  \dA \to \dO
\]
which is a quadratic form on the $\dO$-module $\dA$. 
The general multiplicative group for $\dA$ is defined by 
\[
\GL_{1,\dA} := \dA^\times = \Big(U \mapsto \dA(U)^\times = \{w \in \dA(U) \ ; \ \Nrd(w) \in \dO^\times(U)\}\Big)
\] 
and is an algebraic group over $\bP^1$. The group of rational points of the generic fibre is nothing but
\[
\GL_{1,\dA}(K) = (\dA^\times)_\eta = D^\times,
\]
and the center of $\GL_{1,\dA}$ agrees with $\dO^\times$. We further set 
\[
\SL_{1,\dA}  = \ker(\Nrd: \GL_{1,\dA} \to \dO^\times),
\]
and $\PGL_{1,\dA} =  \GL_{1,\dA}/\dO^\times$, the quotient by the center.
The adjoint action by conjugation defines a group homomorphism
\[
\PGL_{1,\dA} \to \underline{\Aut}(\dA)
\]
which by the Skolem-Noether Theorem is an isomorphism outside the locus $B = \{0,\infty\}$ where $D/K$ ramifies, i.e., an  isomorphism over $U$ as soon as $\dA|_U$ is a sheaf of Azumaya algebras.

\subsection{Global torsion in $\dA^\times$}
\label{sec:globaltorsioninlattice}
Let $S$ be a finite set of places of $\bP^1$ containing the ramified places $B = \{0,\infty\}$ for $D/K$. Let $\fo_{K,S} = \Gamma(\bP^1 \setminus S,\dO)$ be the corresponding ring of $S$-integers. We will abuse notation and evaluate sheaves in $\fo_{K,S}$ when we mean $\bP^1 \setminus S$. 

We will be concerned with the $S$-arithmetic group (see 
Notation~\ref{nota:lamdaSaslatticeinlocallycompactgroup} for $\Lambda_S$ as a lattice)
\[
\Lambda_S := \PGL_{1,\dA}(\fo_{K,S}) = 
\begin{cases} 
\big(\fo_{K,S}\{Z,F\}/(Z^2 = c, F^2 = t, ZF+FZ =0)\big)^\times/\fo_{K,S}^\times & \text{if}~2\nmid q, \\[1ex]
\big(\fo_{K,S}\{Z,F\}/(Z^2+Z = c, F^2 = t, FZ=(Z+1)F)\big)^\times/\fo_{K,S}^\times & \text{if}~2\mid q. 
\end{cases}
\]
Here enters the fact that $\fo_{K,S}$ is a principal ideal domain through the corresponding vanishing in the exact sequence of (nonabelian) \'etale cohomology
\[
\Gm(\fo_{K,S}) \to \GL_{1,\dA}(\fo_{K,S}) \to \PGL_{1,\dA}(\fo_{K,S}) \to \rH^1(\fo_{K,S}, \Gm) = \Pic(\fo_{K,S}) = 0.
\]
Clearly, for $B \subseteq T \subseteq S$ we have an inclusion
\[
\Lambda_T \subseteq \Lambda_S.
\]

For $S = B$ we have the following description of the lattice $\Lambda_B$. Observe that the global sections of $\dA$ 
\[
\rH^0(\bP^1,\dA) = \bF_q \oplus \bF_q \cdot Z =: \bF_q[Z] 
\]
yield a quadratic field extension of $\bF_q$ since $X^2-c$ resp. $X^2+X+c$ is irreducible over $\bF_q$. Hence 
\[
\bF_q[Z]^\times/\bF_q^\times \subseteq \Lambda_B \subseteq \Lambda_S.
\]
Moreover, we also have the image of $F$ defining an element of order $2$ in $\Lambda_B$. Furthermore, $F(-)F^{-1}$ induces the non-trivial Galois automorphism of $\bF_q[Z]$ as an extension of $\bF_q$.

\begin{prop}
\label{prop:global_elements}
The group $\Lambda_B$ is isomorphic to the dihedral group $D_{q+1}$ of order $2(q+1)$ in the form
\[
\Lambda_B = (\bF_q[Z]^\times/\bF_q^\times) \rtimes \langle F \rangle,
\]
with the semidirect product action induced by the Galois action: $F(-)F^{-1} =  q$-Frobenius.
\end{prop}
\begin{proof}
Clearly the right hand side is a subgroup of $\Lambda_B$. For the converse inclusion we have to understand elements of $\dA(\bP^1\setminus B)$ with reduced norm invertible in $H^0(\bP^1\setminus B,\dO)=\bF_q[t^\pm]$.

First we scale by a power of $t$ and may assume that the coefficients of 
\[
w = u+vZ+xF+yFZ
\]
are in $\bF_q[t]$. Let $N_{L/K}: L\to K$ denote the norm map.  For later use, we formulate the computation of $\Nrd(w)$ as a lemma.
\begin{lem}
\label{lem:reducednormformula}
\[
\Nrd(u+vZ+xF+yFZ) = N_{L/K}(u+vZ) - N_{L/K}(x+yZ)t
\]
\end{lem}
\begin{proof}
We abbreviate $\alpha = u+vZ$ and $\beta = x+yZ$. Then we have 
\[
\Nrd(\alpha+ F \beta ) =  (\alpha+ F \beta ) (\ov{\alpha} - \ov{\beta} F ) = \alpha\ov{\alpha} + F \beta \ov{\alpha} - \alpha \ov{\beta} F - F \beta  \ov{\beta} F = N_{L/K}(\alpha) - N_{L/K}(\beta)t.
\qedhere
\]
\end{proof}

Now, $\Nrd(w)$ 
must be of the form $\lambda \cdot t^n$ with $\lambda \in \bF_q$ and $n \in \bN$. Since $u+vZ$ and $\overline{u+vZ}$ are polynomials of the same degree over $\bF_q[Z]$, their product $N_{L/K}(u+vZ)$ is of even degree. The same holds for $N_{L/K}(x+yZ)$. It follows that the reduced norm can only be a monomial if 
\[
N_{L/K}(u+vZ) = 0 \quad \text{ or } \quad N_{L/K}(x+yZ) = 0.
\]
But $N_{L/K}$ does not vanish except at $0$, hence $w = u+vZ$ or $w = (x+yZ)F$. Since $F \in \Lambda_B$ is understood, we may restrict to the case $w = u+vZ$. Now 
\[
\Nrd(w) = N_{L/K}(u+vZ) = \lambda \cdot t^n
\]
and by appropriate scaling with a power of $t$ we may assume $u,v \in \bF_q[t]$ and not both $u(0) = v(0) = 0$. If $u$ or $v$ are non-constant, then $n>0$ and evaluating in $t=0$ yields 
\[
N_{\bF_q[Z]/\bF_q}(u(0)+v(0)Z) = 0,
\]
a contradiction to $u(0)+v(0)Z\neq 0$. Therefore $u,v \in \bF_q$ are constants and thus $\Lambda_B$ is not bigger than claimed.
\end{proof}

\begin{nota}
\label{nota:ds}
The usual structure of a dihedral group being generated by a rotation and a reflection can be obtained for $\Lambda_B$ as follows. The multiplicative group of a finite field is cyclic. Let $d$ denote the image in $\Lambda_B$ of a generator $\delta$ of $\bF_q[Z]^\times$. Then, with $s$ being the image of $F$ we have
\[
\Lambda_B = \langle d,s \ | \ d^{q+1} = s^2 = (sd)^2 = 1 \rangle.
\]
\end{nota}

\begin{rmk}
\label{rmk:representing central 2-torsion in dihedral group}
If $q$ is odd, then $\delta^{(q+1)/2}$ represents a $2$-torsion class in $\bF_q[Z]^\times/\bF_q^\times$. Since $Z^2 = c$, this is therefore
\[
\delta^{(q+1)/2} \equiv Z \pmod{\bF_q^\times}.
\]
\end{rmk}

\subsection{Volumes of $S$-arithmetic groups}
\label{subsec:volume_LambdaS}
The homomorphism of algebraic groups
\[
\SL_{1,\dA} \to \PGL_{1,\dA}
\]
is surjective in the \'etale topology with kernel $\mu_2$. We abbreviate
\[
\tilde{\Lambda}_S = \SL_{1,\dA}(\fo_{K,S})
\]
and obtain the exact sequence of (non-abelian) flat cohomology
\begin{equation}
\label{eq:comparewith_sc_lattice}
1 \to \mu_2(\fo_{K,S}) = \{\pm 1\} \to \tilde{\Lambda}_S \to \Lambda_S \xrightarrow{\delta} \rH^1(\fo_{K,S},\mu_2) = \fo_{K,S}^\times/(\fo_{K,S}^\times)^2.
\end{equation}
\begin{lem}
\label{lem:delta_as_Nrd}
If $a \in \Lambda_S = \PGL_{1,\dA}(\fo_{K,S})$ lifts to $\tilde{a} \in \GL_{1,\dA}(\fo_{K,S})$, then 
\[
\delta(a) = \Nrd(\tilde{a}) \pmod{(\fo_{K,S}^\times)^2}.
\]
\end{lem}
\begin{proof}
This follows by carefully unraveling the definition of the boundary map.
\end{proof}

\begin{prop}
\label{prop:delta_surjective}
For all $S$ with $B \subseteq S$ the map $\delta: \Lambda_S \to \fo_{K,S}^\times/(\fo_{K,S}^\times)^2$ is surjective.
\end{prop}
\begin{proof}
By Lemma~\ref{lem:delta_as_Nrd} we must provide enough reduced norms. First we treat constants, i.e., the image of $\bF_q^\times \to \fo_{K,S}^\times/(\fo_{K,S}^\times)^2$. 
We denote by $N$ the norm map $\bF_q[Z] \to \bF_q$. Norms between finite fields are surjective and for all $u+vZ \in \bF_q[Z]$ the fact
\[
\Nrd(u+vZ) = N(u+vZ) 
\]
shows that constants modulo squares are contained in the image of $\delta$. 

Non-constants modulo squares are provided by $\Nrd(F) = -t$ and the following lemma, again due to the surjectivity of the norm map $N$.
\end{proof}

\begin{lem}
\label{lem:nrd_formula}
For $\alpha \in \bF_q[Z]$ we have
\[
\Nrd(1 + \alpha F) = 1 - N(\alpha)t.
\]
\end{lem}
\begin{proof}
This is a special case of Lemma~\ref{lem:reducednormformula}.
\end{proof}
 
\begin{nota}
For a place $v$ of $K$ we denote the completion in $v$ by $K_v$. For a rational place $t = t_0 \in \bF_q$ we write $K_{t_0}$ for the corresponding completion.
\end{nota}

\begin{nota}
\label{nota:lamdaSaslatticeinlocallycompactgroup}
Let $S_0 = S \setminus B$ be the set of unramified places in $S$. For each $v \in S_0$ we consider the Bruhat-Tits building $\BTT_v$ of $\PGL_2(K_v) \simeq \PGL_{1,\dA}(K_v)$, a regular tree of valency $N(v) + 1$ where $N(v)$ is the norm of the place $v$, i.e., the size of the residue field. We also set 
\[
\PT_{S_0} = \prod_{v \in S_0} \BTT_v.
\]
The arithmetic lattice $\Lambda_S$ acts on $\PT_{S_0}$ through its localisations in all $v \in S_0$ as a discrete cocompact subgroup 
\[
\Lambda_S \inj \prod_{v \in S_0} \PGL_2(K_v).
\]

The choice of the sheaf of maximal orders $\dA$ distinguishes by the local condition at $v \in S_0$ a vertex $\star_v \in \BTT_v$. We denote by $\star_{S_0} \in \PT_{S_0}$ the product of the distinguished vertices in each factor.
The main object of this section is the quotient cube complex
\[
\Lambda_S \backslash \PT_{S_0}.
\]
\end{nota}

We now count vertices $P \in \Lambda_S \backslash \PT_{S_0}$ with weight the reciprocal of the size of the stabiliser $\Lambda_P$ of a representative in $T_{S_0}$. 
We denote this \textbf{orbifold degree} by
\[
N_0(\Lambda_S)
 = \sum_{P  \text{ vertex of } \Lambda_S \backslash \PT_{S_0}} \frac{1}{\#\Lambda_P}.
\]

\begin{prop}
\label{prop:volumeLambdaS}
The orbifold degree of the lattice $\Lambda_S$ is
\[
N_0(\Lambda_S) =  \frac{1}{2(q+1)}.
\]
\end{prop}
\begin{proof}
The orbifold degree relates to the Tits measure $\mu_{\rm Tits}$ used in Prasad's volume formula by a scaling factor that can be determined by comparison with the Euler-Poincar\'e measure $\mu_{{\rm EP},S_0}$ using \cite[\S4]{borel_prasad}  as follows:
\[
N_0(\Lambda_S) \cdot \prod_{v \in S_0} \frac{1-N(v)}{2} = \mu_{{\rm EP},S_0}(\Lambda_S) = \mu_{{\rm Tits}}(\Lambda_S) \cdot \prod_{v \in S_0} \frac{1-N(v)}{1+N(v)}.
\]
Applying Prasad's volume formula \cite[Theorem 3.7]{prasad:volumeformula}  now yields the following, 
\begin{align*}
\mu_{\rm Tits}(\tilde{\Lambda}_S) & = 
q^{-3} \cdot \prod_{v \in S_0} \frac{N(v)^2}{N(v) - 1}  \cdot \prod_{v \in B} \frac{N(v)^2}{N(v) + 1} \cdot \prod_{v \notin S } \frac{1}{1-N(v)^{-2}}  \\
& =  q^{-3} \cdot  \zeta_K(2)  \cdot \prod_{v \in B} (N(v) - 1) \cdot \prod_{v \in S_0} (N(v) + 1)  \\
& =  q^{-3} \cdot \frac{1}{(1-q^{-2})(1-q^{-1})} \cdot (q-1)^2 \cdot \prod_{v \in S_0} (N(v) + 1)  =
\frac{1}{q+1} \cdot \prod_{v \in S_0} (N(v) + 1).
\end{align*}
Hence by Proposition~\ref{prop:delta_surjective}, \eqref{eq:comparewith_sc_lattice} and the function field analogue of Dirichlet's $S$-Unit Theorem,
\begin{align*}
\mu_{\rm Tits}(\Lambda_S) =\mu_{\rm Tits}(\tilde \Lambda_S) \cdot \frac{\#\mu_2(\fo_{K,S})}{\#\fo_{K,S}^\times/(\fo_{K,S}^\times)^2} = \frac{\mu_{\rm Tits}(\tilde \Lambda_S)}{2^{\#S - 1}} = \frac{1}{2(q+1)} \cdot \prod_{v \in S_0} (\frac{N(v) + 1}{2}).
\end{align*}
The formula for $N_0(\Lambda_S)$ follows at once.
\end{proof}

\begin{rmk}
The orbifold degree of $\Lambda_S$ is independent of the set of places $S$ (containing $B$). In the sequel we restrict to places that correspond to $\bF_q$-rational points because in these cases the description of the lattices turns out to be very explicit. Qualitatively, the results also hold without the restriction imposed on $S$, and in each individual case it is easy to derive explicit presentations.
\end{rmk}

\subsection{The generator sets $A_\tau$} 
Let $\tau \in \bF_q^\times = (\bP^1 \setminus B)(\bF_q)$ be a rational point. 
According to Lemma~\ref{lem:nrd_formula} the elements of
\[
A_\tau := \{1+\alpha F \ ; \ \alpha \in \bF_q[Z] \text{ with } N(\alpha) = \tau^{-1}\}
\]
have reduced norm $1- \tau^{-1} t$, and this is a uniformizer at the place $t = \tau$. The composite
\[
A_\tau \subseteq \GL_{1,\dA}(\fo_{K,\{0,\tau,\infty\}}) \to \PGL_{1,\dA}(\fo_{K,\{0,\tau,\infty\}}) = \Lambda_{\{0,\tau,\infty\}} =: \Lambda_\tau
\]
is injective, and we denote its image by 
\[
PA_\tau \subseteq \Lambda_\tau.
\]

\begin{rmk}
\label{rmk:H90}
The group of elements of norm $1$  in $\bF_q[Z]$, the norm-$1$ torus,  is denoted by 
\[
\bT(q) := \bF_q[Z]^{\times, N = 1} := \ker\big(N:\bF_q[Z]^\times \surj \bF_q^\times\big).
\]
We recall the following direct consequence of Hilbert's Theorem $90$: the map 
\[
\bF_q[Z]^\times \surj  \bT(q), \qquad \lambda \mapsto \lambda/\bar \lambda
\]
is surjective.
\end{rmk}

\begin{prop}
\label{prop:stab_act_Atau1}
Let $\tau \in \bF_q^\times$. 
\begin{enumerate}
\item 
The set $PA_\tau$ contains $q+1$ elements. 
\item 
The group $\Lambda_B$ acts transitively by conjugation 
on  $PA_\tau$ in $\Lambda_\tau$. For $\lambda \in \bF_q[Z]^\times$ and $1+ \alpha F \in A_\tau$ we have
\begin{align*}
\lambda (1+\alpha F)\lambda^{-1} & = 1 + \alpha (\lambda/\bar \lambda) F \\
F(1+\alpha F) F^{-1} & = 1+ \bar \alpha F.
\end{align*}
\end{enumerate}
\end{prop}
\begin{proof}
(1) The norm map $N: \bF_q[Z]^\times \to \bF_q^\times$ is surjective, hence $\#A_\tau  = \#\ker(N) = q+1$. The map $A_\tau \to PA_\tau $ is clearly bijective.

(2)
The formulas are straightforward and yield again elements in $A_\tau$ because $N(\lambda/\bar \lambda) = 1$ and $N(\alpha) = N(\bar \alpha)$. 
The action is transitive because of the surjectivity recalled in Remark~\ref{rmk:H90} and the fact that elements of norm $\tau^{-1}$ differ multiplicatively by an element of norm $1$.
\end{proof}

\begin{nota}
We will use the notation $g \mapsto [g]$ for the map 
\[
\GL_{1,\dA}(\fo_{K,S})   \to \PGL_{1,\dA}(\fo_{K,S})
\]
if it becomes necessary to distinguish between group elements and representatives modulo the center.
\end{nota}

\begin{nota}
\label{nota:a_tau}
Recall that $s = [F] \in \PGL_{1,\dA}(\fo_{K,S})$ is the image of $F$ and $d = [\delta] \in \Lambda_B$ is the image of a fixed generator $\delta$ of $\bF_q[Z]^\times$.
Let $\tau \in \bF_q^\times$. We define an element $1 + \alpha_0 F \in A_\tau$ and its image $a_\tau \in \Lambda_\tau$ depending on whether $\tau$ is a square in $\bF_q$ or not.
\begin{itemize}
\item If $\tau$ is a square  in $\bF_q$ (note that this case includes the case of even $q$), then we fix a square root $\alpha_0$ of $\tau^{-1}$:
\[
\tau^{-1} = (\alpha_0)^2.
\]
Clearly then $N(\alpha_0) = \tau^{-1}$ and $1+ \alpha_0 F \in A_\tau$ yields an element $a_\tau \in PA_\tau$.
\item
Since $\delta$ generates $\bF_q[Z]^\times$ and the norm map is surjective, the norm $N(\delta) = \delta^{q+1}$ generates $\bF_q^\times$. As a generator therefore $\delta^{q+1}$ is not a square in $\bF_q$.
If $\tau$ is not a square in $\bF_q$, then $\delta^{q+1} \tau^{-1}$ is a square, say 
\[
(\alpha_1)^2 = \delta^{q+1} \tau^{-1},
\]
and we set $\alpha_0 = \alpha_1 \delta^{-1}$. This is an element of norm
\[
N(\alpha_0) =(\alpha_1)^2 \delta^{-1-q} = \tau^{-1}
\]
and thus $1+ \alpha_0 F \in A_\tau$ yields an element $a_\tau \in PA_\tau$.
\end{itemize}
\end{nota}

\begin{prop}
\label{prop:StabAtau_relations}
Let $\tau \in \bF_q^\times$. 
\begin{enumerate}
\item 
\label{propitem:StabAtau_relations1}
We have 
\[
PA_\tau = \{ d^i a_\tau d^{-i} \ ; \ 0 \leq i \leq q\}.
\]
\item
\label{propitem:StabAtau_relations2}
For all $a = [1+\alpha F] \in PA_\tau$ we have
\[
a^{-1} =  [1 - \alpha F] =
\begin{cases}
d^{(q+1)/2} a d^{-(q+1)/2} & \text{if $2 \nmid q$}, \\[1ex]
a & \text{if $2 \mid q$}.
\end{cases}
\]
\item
\label{propitem:StabAtau_relations3}
If $\tau$ is a square in $\bF_q$, then 
\[
sa_\tau s = a_\tau,
\] 
and,  if $\tau$ is not a square in $\bF_q$, then 
\[
(sd)a_\tau (sd)^{-1}  = a_\tau.
\]
\end{enumerate}
\end{prop}
\begin{proof}
Assertion \ref{propitem:StabAtau_relations1} is clear from Proposition~\ref{prop:stab_act_Atau1}.

\ref{propitem:StabAtau_relations2} In the case $2 \nmid q$, the central $2$-torsion element $d^{(q+1)/2}$ is represented by $Z$ by Remark~\ref{rmk:representing central 2-torsion in dihedral group}. Then the claim follows from 
\[
[Z][1+\alpha F] [Z^{-1}] [1+ \alpha F] = [1- \alpha F][1+ \alpha F] = [\Nrd(1+\alpha F)] = 1.
\]
In the case $2 \mid q$ we compute
\[
[1 + \alpha F]^2 = [1 + 2\alpha F  + \alpha \ov{\alpha} F^2] = [1 + N(\alpha) t] = 1.
\]

\ref{propitem:StabAtau_relations3} If $\tau$ is a square, then $a_\tau$ is represented by $1 + \alpha_0 F$ with $\alpha_0 \in \bF_q$, and thus
\[
sa_\tau s = [F(1+\alpha_0F)F^{-1}] = [1+\alpha_0 F]  = a_\tau.
\]
If $\tau$ is not a square, then $a_\tau = [1 + \alpha_1 \delta^{-1} F]$ with $\alpha_1 \in \bF_q$, and thus 
\[
(sd)a_\tau(sd)^{-1} = [F\delta(1+\alpha_1 \delta^{-1} F) \delta^{-1} F^{-1}] = [F(1+\alpha_1 {\bar \delta}^{-1} F)F^{-1}] = [1+\alpha_1 \delta^{-1} F] = a_\tau. \qedhere
\]
\end{proof}

\subsection{The $t$-adic filtration} 
We now extract valuable information by working locally at the ramified place $t=0$ of the quaternion algebra $D/K$. At $t=0$, the completion is the field of Laurent power series
\[
K_0 = \bF_q((t)).
\]
Being ramified, the base change of $D$ to $K_0$ is still a skew field:
\[
\cD_0 := D \otimes_K K_0 = 
\begin{cases}
\left(\frac{c,t}{\bF_q((t))}\right) & \text{ if $2 \nmid q$} \\[1ex]
\left[\frac{c,t}{\bF_q((t))}\right) \ ) & \text{ if $2 \mid q$}.
\end{cases}
\]
As a skew field over a local field, $\cD_0$ has a unique maximal order $\fO_{D,0} \subset \cD_0$. Because of $F^2 = t$, the corresponding valuation  has $F$ as uniformizer generating the unique two-sided maximal ideal $\fm_0 \subseteq \fO_{D,0}$. The `cotangent space' 
\[
\fm_0/\fm_0^2  = F \fO_{D,0} /F^2 \fO_{D,0}
\]
is preserved by conjugation with an arbitrary element. Hence we obtain an $\bF_q$-linear action of $D^\times \subseteq \cD_0^\times$ on the $\bF_q$-vector space
$\fm_0/\fm_0^2$. The map 
\[
\bF_q[Z] \to \fm_0/\fm_0^2, \qquad u+vZ \mapsto (u+vZ)F + \fm_0^2
\]
allows one to interpret this action as an $\bF_q$-linear action on $\bF_q[Z]$ that is computed by 
conjugating in $\cD_0$ (or $D$) and working modulo linear and higher order power series in $t$.  
Since the center of $\cD_0$ acts trivially, we can deduce a homomorphism
\begin{equation}
\label{eq:action on Zariski cotangent space at 0}
\rho: \Lambda_S \to \Aut_{\bF_q\text{-linear}}(\bF_q[Z]).
\end{equation}

\begin{prop}
\label{prop:formula_rho}
Let $\tau \in \bF_q^\times$ and let $[1 + \alpha F] \in PA_\tau$. Then the following holds.
\begin{enumerate}
\item 
\[
\rho([1 + \alpha F]) = 1.
\]
\item 
For $\lambda \in \bF_q[Z]$ we have
\[
\rho([\lambda]) = \Big(u+vZ \mapsto \lambda/\bar \lambda \cdot (u+vZ)\Big)
\]
\item
The map $\rho(F)$ is the $q$-Frobenius 
\[
\rho([F]) = 
\begin{cases}
\big(u+vZ \mapsto u-vZ\big) & \text{if $2\nmid q$,} \\[1ex]
\big(u+vZ \mapsto u+v+vZ\big) & \text{if $2\mid q$.}
\end{cases}
\]
\end{enumerate}
\end{prop}
\begin{proof}
Straightforward computations. For (1) we use that $F^2 = t$ vanishes in $\fm_0 /\fm_0^2$
\end{proof}

\begin{cor}
The map $\rho$ induces an isomorphism of $\Lambda_B$ with the normaliser of the non-split torus $\bT(q) \subseteq \SL(\bF_q[Z])$, where $\SL(\bF_q[Z])$ denotes the $\bF_q$-linear automorphisms of $\bF_q[Z]$ of determinant $1$.
\end{cor}
\begin{proof}
Immediate from Proposition~\ref{prop:formula_rho}.
\end{proof}

\begin{nota}
Let $S \subseteq \bP^1$ be a finite set of $\bF_q$-rational places containing $B = \{0,\infty\}$. We set $S_0 = S \setminus B$ and 
\[
\Gamma_S :=  \langle PA_\tau \ ; \ \tau \in S_0 \rangle \subseteq \Lambda_S
\]
for the subgroup in $\Lambda_S$ generated by all $[1+\alpha F]$ with $N(\alpha)^{-1} = \tau \in S_0$.
\end{nota}

\begin{prop}
\label{prop:p_res_finite}
Let $S \subseteq \bP^1$ be a finite set of $\bF_q$-rational places containing $B = \{0,\infty\}$.  Then the group $\Gamma_S$ is $p$-residually finite.
\end{prop}
\begin{proof}
The generators $[1+\alpha F]$ are images of $1$-units with respect to the valuation of $\cD_0$.  The inclusion 
\[
\Gamma_S \subseteq \im\big(1+\fm_0 \to \PGL_{1,\dA}(K_0)\big)
\]
and the fact that the multiplicative group $1+\fm_0$ is a pro-$p$ group prove the result.
\end{proof}

\subsection{Presentation of $\Lambda_S$} 
We are now ready to describe $\Lambda_S$ in terms of generators and relations. We start with stabiliser and orbit of the standard vertex.

\begin{prop}
\label{prop:GammaS}
Let $S \subseteq \bP^1$ be a finite set of $\bF_q$-rational places containing $B = \{0,\infty\}$. \begin{enumerate}
\item
\label{propitem:GammaS1}
The stabiliser of $\star_{S_0} \in \PT_{S_0}$ for the action of $\Lambda_S$ is the dihedral group $\Lambda_B$ of order $2(q+1)$.
\item 
\label{propitem:GammaS2}
The group $\Lambda_S$ acts transitively on the vertices of $\PT_{S_0}$.
\item 
\label{propitem:GammaS3}
The group $\Gamma_S$ acts transitively on the vertices of $\PT_{S_0}$.
\end{enumerate}
\end{prop}
\begin{proof}
(1) The stabiliser in $\Lambda_S$ of the distinguished vertex $\star_{S_0}$ consists of those elements that are also integral with respect to $\dA$ at all $\tau \in S_0$, hence this is $\PGL_{1,\dA}(\fo_{K,\{0,\infty\}}) = \Lambda_B$, 
and  due to Proposition~\ref{prop:global_elements} a dihedral group of the respective order.

(2) Since by Proposition~\ref{prop:volumeLambdaS} and \ref{propitem:GammaS1} we have
\[
N_0(\Lambda_S) = \frac{1}{2(q+1)} = \frac{1}{\#\Lambda_B},
\]
there is no space for more vertices in  $\Lambda_S\backslash \PT_{S_0}$. Therefore $\Lambda_S$ acts vertex transitively on $\PT_{S_0}$. 

(3) Let $\tau$ be in $S_0$. The elements $[1+\alpha F] \in PA_\tau$ stabilise $\star_v$ for all $v \in S_0$, $v \not= \tau$. Moreover, since by Lemma~\ref{lem:nrd_formula} the reduced norm of $1+\alpha F$ is $1-t/\tau$, hence of $\tau$-valuation $1$ and integral, the standard vertex $\star_{S_0}$ is moved by $[1+\alpha F] \in PA_\tau$ to a neighbour in $\tau$-direction. 

Let $\rho : \Lambda_S \to  \Aut_{\bF_q\text{-linear}}(\bF_q[Z])$ be the homomorphism  from \eqref{eq:action on Zariski cotangent space at 0}. Then $\rho$ is injective on $\Lambda_B$ and trivial for $\Gamma_S$. Therefore,  applying $\rho$ shows that for $[1+\alpha F] \not= [1 + \beta F] \in PA_\tau$ we have
\[
[1+\alpha F]^{-1} [1+\beta F] \notin \Lambda_B.
\]
It follows that $[1+\alpha F] \mapsto [1+\alpha F]\star_{S_0}$ is a bijective map from $PA_\tau$ to the set of all $q+1$ neighbours of $\star_{S_0}$ in  $\tau$-direction. 
Hence  all neighbours of $\star_{S_0}$ are in its $\Gamma_S$-orbit. Therefore the orbit $\Gamma_S\star_{S_0}$ contains all neighbours of its vertices. It follows that  $\Gamma_S$ also acts vertex transitively on $\PT_{S_0}$. 
\end{proof}

\begin{cor}
\label{cor:GammaS}
Let $S \subseteq \bP^1$ be a finite set of $\bF_q$-rational places containing $B = \{0,\infty\}$ and $S_0 = S \setminus B$. 
\begin{enumerate}
\item
\label{coritem:GammaS1}
The group $\Lambda_S$ is generated by $\Lambda_B$ and all $PA_\tau$ for $\tau \in S_0$, more precisely
\[
\Lambda_S = \Gamma_S \Lambda_B
\]
as a product of sets.
\item
\label{coritem:GammaS2}
The group $\Lambda_S$ is generated by $d$, $s$ and $a_\tau$ for all $\tau \in S_0$. 
\item
\label{coritem:GammaS3}
The image of $\rho$ is isomorphic to $\Lambda_B$. Identifying this image with $\Lambda_B$ we have a short exact sequence
\[
1 \to \Gamma_S \to \Lambda_S \xrightarrow{\rho} \Lambda_B \to 1.
\]
\item 
\label{coritem:GammaS4}
The subgroup $\Gamma_S$ is a normal subgroup of index
\[
(\Lambda_S : \Gamma_S) = 2(q+1).
\]
\end{enumerate}
\end{cor}
\begin{proof}
(1) It follows from Proposition~\ref{prop:GammaS} \ref{propitem:GammaS3} that $\Lambda_S$ is generated by $\Gamma_S$ and the stabiliser $\Lambda_B$ in the form claimed by \ref{coritem:GammaS1}. Assertion (2) follows from assertion \ref{coritem:GammaS1} and Proposition~\ref{prop:StabAtau_relations}. 

(3) The claim on the image of $\rho$ follows from \ref{coritem:GammaS1} and the formulae of 
Proposition~\ref{prop:formula_rho}. The kernel of $\rho$ contains $\Gamma_S$ by loc.~cit., and agrees with it because \ref{coritem:GammaS1} says
\[
2(q+1) = \#\Lambda_B \geq (\Lambda_S : \Gamma_S) \geq (\Lambda_S : \ker(\rho) )  = \#\rho(\Lambda_S) = \#\Lambda_B  = 2(q+1).
\]

The assertion in (4)  follows at once from the proof of \ref{coritem:GammaS3}. 
\end{proof}

\begin{cor}
\label{cor:GammaSsimplytransitiv}
Let $S \subseteq \bP^1$ be a finite set of $\bF_q$-rational places containing $B = \{0,\infty\}$ and $S_0 = S \setminus B$.  Then 
\[
N_0(\Gamma_S) = 1,
\]
and  $\Gamma_S$ acts simply transitively on the vertices of $\PT_{S_0}$.
\end{cor}
\begin{proof}
The orbifold degree follows from the  multiplicativity 
$N_0(\Gamma_S) = (\Lambda_S : \Gamma_S) N_0(\Lambda_S)$, 
Proposition~\ref{prop:volumeLambdaS} and 
Corollary~\ref{cor:GammaS} \ref{coritem:GammaS4}. 

The stabiliser in the standard vertex $\star_\tau \in \BTT_\tau$ for the action of $\Gamma_\tau$ on $\BTT_\tau$ equals $\Gamma_S \cap \Lambda_B=1$ by Proposition~\ref{prop:GammaS}
\ref{propitem:GammaS1} and Corollary~\ref{cor:GammaS} \ref{coritem:GammaS3}. As moreover $\Gamma_\tau$ acts vertex transitively by Proposition~\ref{prop:GammaS}
\ref{propitem:GammaS3}, all vertex stabilisers are trivial. 
\end{proof}

\subsubsection{The case $q$ odd}
We first deal with the  case $S_0 = \{\tau\}$ of rank $1$.

\begin{thm}
\label{thm:structure_Lambda_tau}
Let $\tau \in \bF_q^\times$ and set $\Gamma_\tau := \langle PA_\tau \rangle \subseteq \Lambda_\tau$.
\begin{enumerate}
\item
The group $\Gamma_\tau$ is a free group of rank $\frac{q+1}{2}$ on the set $PA_\tau$ considered as a set of generators and their inverses with the involution $a \mapsto a^{-1}$ given by 
\[
[1+\alpha F]^{-1} = [1 - \alpha F].
\]

\item Moreover, $\Gamma_\tau$ is a normal subgroup in $\Lambda_\tau$ with complement $\Lambda_B$, so that we have a semi-direct product:
\[
\Lambda_\tau = \Gamma_\tau \rtimes \Lambda_B = \left\langle d, s, a_\tau \ \left| \begin{array}{c} \ d^{q+1} = s^2  = (sd)^2 = (d^{(q+1)/2}a_\tau)^2 = 1, \\
\begin{cases}
sa_\tau  = a_\tau  s & \text{ if $\tau$ is a square in $\bF_q$}, \\
(sd)a_\tau =a_\tau (sd) & \text{ if $\tau$ is not a square in $\bF_q$}.
\end{cases} \end{array}
\right.\right\rangle.
\]
Here $d,s$ are as in Notation~\ref{nota:ds} and $a_\tau \in PA_\tau$ is as in 
Notation~\ref{nota:a_tau}.
\end{enumerate}
\end{thm}
\begin{proof}
(1) We must show that $\Gamma_\tau$ acts freely on the tree $\BTT_\tau$ without inversion in order to deduce the freeness of $\Gamma_\tau$ from Bass--Serre theory \cite{serre:trees}. Triviality of vertex stabilisers was established in Corollary~\ref{cor:GammaSsimplytransitiv}. Due to $\Gamma_\tau$ acting vertex transitively by Proposition~\ref{prop:GammaS} \ref{propitem:GammaS3}, if there exist an edge on which $\Gamma_\tau$ acts by inversion, then there also is an edge starting in the standard vertex $\star_\tau$. In the proof of 
Proposition~\ref{prop:GammaS} we established that the elements of $\Gamma_\tau$ mapping $\star_\tau$ to its neighbours are exactly given by the elements of $PA_\tau$. So an inversion of such an edge means one of the elements $[1 + \alpha F] \in PA_\tau$ must be of order $2$. 
This contradicts that $\Gamma_\tau$ is $p$-residually finite 
according to Proposition~\ref{prop:p_res_finite}. 

More precisely, Bass--Serre theory shows now that $\Gamma_\tau$ is freely generated by $PA_\tau$ in the sense of generators with their inverses. The involution $a \mapsto a^{-1}$ on $PA_\tau$ was computed in Proposition~\ref{prop:StabAtau_relations} \ref{propitem:StabAtau_relations2}.

(2) That $\Gamma_\tau$ is a normal subgroup with complement $\Lambda_B$ was proven in Corollary~\ref{cor:GammaS}. The explicit finite presentation follows from the freeness of $\Gamma_\tau$ and the formulas from Proposition~\ref{prop:StabAtau_relations} 
that describe the action of 
\[
\Lambda_B = \langle d,s \ | \ d^{q+1} = s^2  = (sd)^2 = 1\rangle
\]
on $PA_\tau = \{ d^i a_\tau d^{-i} \ ; \ 0 \leq i \leq q\}$ together with the fact that $d^{(q+1)/2} a_\tau d^{-(q+1)/2} = a_\tau^{-1}$.
\end{proof}

Now we address the case of higher rank. 

\begin{nota}
\label{nota:zeta}
For $\alpha, \beta \in \bF_q[Z]^\times$ with $N(\alpha) \not= N(\beta)$, then $1 + \alpha/\beta \not= 0$ and we set
\[
\zeta_{\alpha}(\beta) = \frac{1 + \alpha/\beta}{1 + \bar \alpha/ \bar \beta}   \in \bT(q).
\]
When defined, we have
\[
\zeta_{\alpha}(\beta) \beta  = \frac{\alpha+\beta}{1 + \bar \alpha/ \bar \beta} = \bar \beta \cdot \frac{\alpha+\beta}{\bar \alpha +  \bar \beta}.
\]
\end{nota}

\begin{prop}
For $\tau \not= \sigma \in \bF_q^\times$ and $1 + \alpha F \in A_\tau$, and $1 + \beta F \in A_\sigma$ we have in $\GL_{1,\dA}(K)$
\[
(1 + \alpha F)(1+ \beta F) = (1 + \zeta_{\alpha}(\beta) \beta F)(1  + \zeta_\beta(\alpha) \alpha F).
\]
\end{prop}
\begin{proof}
This is a simple calculation:
\begin{align*}
(1 + \zeta_{\alpha}(\beta) \beta  F)(1 + \zeta_\beta(\alpha) \alpha  F) & = 1    + \zeta_{\alpha}(\beta) \beta  F + \zeta_\beta(\alpha) \alpha  F +  \zeta_{\alpha}(\beta) \beta \cdot \ov{ \zeta_\beta(\alpha) \alpha} F^2 \\
& = 1  +  \Big(\bar \beta \cdot \frac{\alpha+\beta}{\bar \alpha +  \bar \beta} + \bar \alpha \cdot \frac{\alpha+\beta}{\bar \alpha +  \bar \beta} \Big) F + \alpha  \bar \beta F^2 \\
& = 1  +  (\alpha + \beta) F + \alpha  F  \beta F   = (1 + \alpha F)(1+ \beta F) . \qedhere
\end{align*}
\end{proof}

Next we can describe a presentation of $\Gamma_S$.

\begin{thm}
\label{thm:structure_Gamma_S}
Let $S \subseteq \bP^1$ be a finite set of $\bF_q$-rational places containing $B = \{0,\infty\}$ and $S_0 = S \setminus B$. 
\begin{enumerate}
\item 
The group $\Gamma_S$ has the following presentation:
\begin{itemize}
\item
$\Gamma_S$ is generated by $\bigcup_{\tau \in S_0} PA_\tau$,
\item
subject to the following relations:
\begin{enumerate}
\item[(i)]
for all $\tau \in S_0$ and $[1 + \alpha F] \in PA_\tau$, we have
\[
[1 + \alpha F] \cdot [1- \alpha F] = 1,
\]
\item[(ii)]
for all $\tau \not= \sigma \in S_0$ and $[1 + \alpha F] \in PA_\tau$, and $[1 + \beta F] \in PA_\sigma$, we have
\[
[1 + \alpha F] \cdot [1+ \beta F] = [1 + \zeta_{\alpha}(\beta) \beta F] \cdot [1  + \zeta_\beta(\alpha) \alpha F].
\]
\end{enumerate}
\end{itemize}
Here $\zeta_{\alpha}(\beta)$ is as in Notation~\ref{nota:zeta}.
\item
The group $\Gamma_S$ acts on the cube complex $\PT_{S_0}$ without fixing a cell. In particular, $\Gamma_S$ is torsion-free.
\end{enumerate}
\end{thm}
\begin{proof}
(1) Since $\Gamma_S$ acts simply transitively on the vertices of $\PT_{S_0}$ and the generators in $\bigcup_{\tau \in S_0} PA_\tau$ correspond exactly to the neighbours of the standard vertex $\star_{S_0}$, the $1$-skeleton of $\PT_{S_0}$ is the Cayley graph for $\Gamma_S$ with respect to the generating set $\bigcup_{\tau \in S_0} PA_\tau$. 

The relations are obtained as cycles in the Cayley graph, which are products of backtracking and cycles around $2$-dimensional faces in $\PT_{S_0}$. All $2$-dimensional faces of $\PT_{S_0}$ are translates by $\Gamma_S$ of $2$-dimensional faces with $\star_{S_0}$ as one of their corners. Therefore it is enough for a presentation of $\Gamma_S$ to consider relations obtained by backtracking of an edge from $\star_{S_0}$ and cycles around $2$-dimensional faces of $\PT_{S_0}$ starting from $\star_{S_0}$. These are precisely the relations claimed in the theorem.

\smallskip

(2) Since $\PT_{S_0}$ is a CAT(0) space, any finite subgroup of $\Gamma_S$ stabilises a cell of $\PT_{S_0}$ by the Bruhat--Tits fixed point lemma \cite[II.2.8(1)]{bridson_haeflinger}. So the first claim implies the second. 

Because we already know that $\Gamma_S$ acts simply transitively on the vertices of $\PT_{S_0}$, the only possibility for fixing a cell is by central symmetry of a face of the cell,  and this leads to an element of order $2$. This contradicts that $\Gamma_S$ is $p$-residually finite according to Proposition~\ref{prop:p_res_finite}. 
\end{proof}

\begin{cor} \label{cor:structure_Gamma_S}
Let $S \subseteq \bP^1$ be a finite set of $\bF_q$-rational places containing $B = \{0,\infty\}$ and $S_0 = S \setminus B$. 
The group $\Lambda_S$ is a semi-direct product $\Lambda_S = \Gamma_S \rtimes \Lambda_B$, and we have the following presentation:
\begin{itemize}
\item
$\Lambda_S$ is generated by $d, s$ and $a_\tau$ for each $\tau \in S_0$,
\item
For all $\tau \in S_0$ we have $PA_\tau = \{ d^i a_\tau d^{-i} \ ; \ 0 \leq i \leq q\} \subseteq \Lambda_S$,
\item
relations I: $d^{q+1} = s^2 = (sd)^2 = 1$, and for all $\tau \in S_0$ we have $(d^{(q+1)/2}a_\tau)^2 = 1$.
\item
relations II:
\begin{align*}
a_\tau & = sa_\tau s & \text{ if $\tau$ is a square in $\bF_q$}, \\
a_\tau & = (sd)a_\tau (sd)^{-1} &  \text{ if $\tau$ is not a square in $\bF_q$}. 
\end{align*}
\item
relations III: for all $\tau \not= \sigma \in S_0$ and $[1 + \alpha F] \in PA_\tau$, and $[1 + \beta F] \in PA_\sigma$ we have
\[
[1 + \alpha F] \cdot [1+ \beta F] = [1 + \zeta_{\alpha}(\beta) \beta F] \cdot [1  + \zeta_\beta(\alpha) \alpha F].
\]
\end{itemize}
Here $d,s$ are as in Notation~\ref{nota:ds}, the element $a_\tau \in PA_\tau$ is as in 
Notation~\ref{nota:a_tau}, and $\zeta_{\alpha}(\beta)$ is as in 
Notation~\ref{nota:zeta}.
\end{cor}
\begin{proof}
Corollary~\ref{cor:GammaS} shows the structure of a semi-direct group $\Lambda_S = \Gamma_S \rtimes \Lambda_B$. The  relations describing the action of $\Lambda_B$ on the generators $\bigcup_{\tau \in S_0} PA_\tau$ are proven in
Proposition~\ref{prop:StabAtau_relations}, and the relations for $\Gamma_S$ are dealt with in Theorem~\ref{thm:structure_Gamma_S}.
\end{proof}

\subsubsection{The case $q$ even}
As in the case $q$ odd, we begin with the case of rank 1.
\begin{thm}
\label{thm:structure_Lambda_tau_even_q}
Let $\tau \in \bF_q^\times$ and $\Gamma_\tau := \langle PA_\tau \rangle \subseteq \Lambda_\tau$.
\begin{enumerate}
\item
\label{thmitem:Lambda_tau_even_q2}
The group $\Gamma_\tau$ has the following presentation:
\[
\Gamma_\tau = \bigl\langle PA_\tau \ | \ [1+\alpha F]^2=1 \ \text{for all $[1+\alpha F]\in PA_\tau$} \bigr\rangle
\]
\item
\label{thmitem:Lambda_tau_even_q3}
The group $\Gamma_\tau$ is a normal subgroup in $\Lambda_\tau$ with complement $\Lambda_B$, so that we have a semi-direct product:
\[
\Lambda_\tau = \Gamma_\tau \rtimes \Lambda_B = \bigl\langle d, s, a_\tau \ | \ d^{q+1} = s^2  = (sd)^2 = a_\tau^2 = 1, \ sa_\tau  = a_\tau  s\bigr\rangle.
\]
Here $d,s$ are as in Notation~\ref{nota:ds} and $a_\tau \in PA_\tau$ is as in Notation~\ref{nota:a_tau}.
\end{enumerate}
\end{thm}
Note that for $q$ even, we always have $a_\tau=1+\alpha_0F$ for a uniquely determined $\alpha_0\in\bF_q$ with $\alpha_0^2=\tau^{-1}$.
\begin{proof}
For the first part, we observe that by Corollary \ref{cor:GammaSsimplytransitiv} the action of $\Gamma_\tau$ on $\BTT_\tau$ has no fixed vertices, but $\Gamma_\tau$ acts with inversion of edges. Each edge attached to the standard vertex $\star_\tau$ is precisely stabilised by $\langle [1+\alpha F] \rangle\cong\bZ/2\bZ$ for some $[1+\alpha F]\in PA_\tau$. Hence the action of $\Gamma_\tau$ on the barycentric subdivision of $\BTT_\tau$ has no inversion but each neighbour of $\star_\tau$ (which was the midpoint of the corresponding edge in $\BTT_\tau$) has $\langle [1+\alpha F] \rangle$ with the corresponding $[1+\alpha F]\in PA_\tau$ as stabiliser. Hence by Bass-Serre theory, cf. \cite[I.5.4, Thm.13]{serre:trees}, $\Gamma_\tau$ is an amalgam of $\langle [1+\alpha F] \rangle$, where $[1+\alpha F]$ runs over all elements of $PA_\tau$. This implies the presentation of $\Gamma_\tau$ as desired.

The explicit presentation of $\Gamma_\tau$ as in \ref{thmitem:Lambda_tau_even_q3} follows from Proposition~\ref{prop:StabAtau_relations} that describes the action of 
\[
\Lambda_B = \langle d,s \ | \ d^{q+1} = s^2  = (sd)^2 = 1\rangle
\]
on $PA_\tau = \{ d^i a_\tau d^{-i} \ ; \ 0 \leq i \leq q\}$.
\end{proof}

Next we can describe a presentation of $\Gamma_S$.

\begin{thm}
\label{thm:structure_Gamma_S_even_q}
Let $S \subseteq \bP^1$ be a finite set of $\bF_q$-rational places containing $B = \{0,\infty\}$ and $S_0 = S \setminus B$. 
\begin{enumerate}
\item
The group
\[
\Gamma_S' := \left\langle aa' \ ; \ a,a' \in PA_\tau, \ \tau \in S_0 \right\rangle
\]
is a normal subgroup of $\Gamma_S$ with quotient $\{\pm1\}^{S_0}$ . Moreover, $\Gamma_S'$ acts on the cube complex $\PT_{S_0}$ without fixing a cell. In particular, $\Gamma_S'$ is torsion-free. 
\item 
The group $\Gamma_S$ has the following presentation:
\begin{itemize}
\item
$\Gamma_S$ is generated by $\bigcup_{\tau \in S_0} PA_\tau$,
\item
subject to the following relations:
\begin{enumerate}
\item[(i)]
for all $\tau \in S_0$ and $[1 + \alpha F] \in PA_\tau$, we have
\[
[1 + \alpha F]^2 = 1,
\]
\item[(ii)]
for all $\tau \not= \sigma \in S_0$ and $[1 + \alpha F] \in PA_\tau$, and $[1 + \beta F] \in PA_\sigma$, we have
\[
[1 + \alpha F] \cdot [1+ \beta F] = [1 + \zeta_{\alpha}(\beta) \beta F] \cdot [1  + \zeta_\beta(\alpha) \alpha F].
\]
\end{enumerate}
\end{itemize}
Here $\zeta_{\alpha}(\beta)$ is as in Notation~\ref{nota:zeta}.
\end{enumerate}
\end{thm}
\begin{proof}
(1) Consider the group homomorphism $\Gamma_S \to \{\pm1\}^{S_0}$, $[a] \mapsto (e_\tau(a))_{\tau\in S_0}$, where
\[
e_\tau(a) := (-1)^{\ord_\tau(\Nrd(a))}.
\]
It is easily seen that this is surjective and has $\Gamma_S'$ as kernel. Hence $\Gamma_S'$ is a normal subgroup of $\Gamma_S$ with quotient $\{\pm1\}^{S_0}$. To see that this acts on the cube complex $\PT_{S_0}$ without fixing a cell, observe that since $\Gamma_S$ acts simply transitively on the vertices of $\PT_{S_0}$, the only possibility for fixing a cell is by central symmetry of a face of a cell. Under this action, each vertex is mapped to another vertex with odd distances along all the directions of this cell. This cannot be the case for those elements from $\Gamma_S'$ since all the generators have reduced norm $(1-t/\tau)^2$ for some $\tau\in S_0$, meaning that they send each vertex to another vertex of even distance along all the directions. Consequently, $\Gamma_S'$ must be torsion-free by the Bruhat-Tits fixed point lemma.

(2) This can be done in the same way as Theorem \ref{thm:structure_Gamma_S}.
\end{proof}

\begin{cor}
Let $S \subseteq \bP^1$ be a finite set of $\bF_q$-rational places containing $B = \{0,\infty\}$ and $S_0 = S \setminus B$. 
The group $\Lambda_S$ is a semi-direct product $\Lambda_S = \Gamma_S \rtimes \Lambda_B$, and we have the following presentation:
\begin{itemize}
\item
$\Lambda_S$ is generated by $d, s$ and $a_\tau$ for each $\tau \in S_0$,
\item
For all $\tau \in S_0$ we have $PA_\tau = \{ d^i a_\tau d^{-i} \ ; \ 0 \leq i \leq q\} \subseteq \Lambda_S$,
\item
relations I: $d^{q+1} = s^2 = (sd)^2 = 1$, and for all $\tau \in S_0$ we have $a_\tau^2 = 1$.
\item
relations II: $a_\tau = sa_\tau s$.
\item
relations III: for all $\tau \not= \sigma \in S_0$ and $[1 + \alpha F] \in PA_\tau$, and $[1 + \beta F] \in PA_\sigma$ we have
\[
[1 + \alpha F] \cdot [1+ \beta F] = [1 + \zeta_{\alpha}(\beta) \beta F] \cdot [1  + \zeta_\beta(\alpha) \alpha F].
\]
\end{itemize}
Here $d,s$ are as in Notation~\ref{nota:ds}, the element $a_\tau \in PA_\tau$ is as in 
Notation~\ref{nota:a_tau}, and $\zeta_{\alpha}(\beta)$ is as in 
Notation~\ref{nota:zeta}.
\end{cor}
\begin{proof}
This is done in the same way as Corollary \ref{cor:structure_Gamma_S}.
\end{proof}

\subsection{A concrete model of $\Lambda_S$} 
\label{sec:concretemodel}
Let $q$ be a prime power. 
We describe the presentation of $\Lambda_S$ in terms of finite fields only.  Let 
\[
\delta \in \bF_{q^2}^\times
\]
be a generator of the multiplicative group of the field with $q^2$ elements. If $i,j \in \bZ/(q^2-1)\bZ$ are 
\[
i \not\equiv j \pmod{q-1},
\]
then $1+\delta^{j-i} \not= 0$, since otherwise
\[
1 = (-1)^{q+1} = \delta^{(j-i)(q+1)} \not= 1,
\]
a contradiction. Then there is a unique $x_{i,j} \in \bZ/(q^2-1)\bZ$ with 
\[
\delta^{x_{i,j}} = 1  + \delta^{j-i}.
\]
With these $x_{i,j}$ we set $y_{i,j} := x_{i,j} + i - j$, so that 
\[
\delta^{y_{i,j}} = \delta^{x_{i,j} + i - j} = (1  + \delta^{j-i}) \cdot \delta^{i-j} = 1  + \delta^{i-j}.
\]
We moreover set  
\begin{align*}
l(i,j) & := i -  x_{i,j}(q-1), \\
k(i,j) & := j -  y_{i,j}(q-1).
\end{align*}

If $\alpha = \delta^i$ and $\beta = \delta^j$, then 
\begin{equation}
\label{eq:concreterelation1}
\delta^{k(i,j)}  = \delta^{j -  y_{i,j}(q-1)} = \delta^j (1+ \delta^{i-j})^{1-q} =  \frac{\delta^i + \delta^j}{(1+ \delta^{i-j})^q} =  \frac{\delta^i + \delta^j}{(\delta^i + \delta^{j})^q} \cdot \delta^{jq} = \zeta_{\alpha}(\beta) \beta,
\end{equation}
and
\begin{equation}
\label{eq:concreterelation2}
\delta^{l(i,j)}  = \delta^{i-  x_{i,j}(q-1)} = \delta^i (1+ \delta^{j-i})^{1-q} =  \frac{\delta^i + \delta^j}{(1+ \delta^{j-i})^q} =  \frac{\delta^i + \delta^j}{(\delta^i + \delta^{j})^q} \cdot \delta^{iq} = \zeta_{\beta}(\alpha) \alpha.
\end{equation}

Let now $M \subseteq \bZ/(q^2-1)\bZ$ be a union of cosets under $(q-1)\bZ/(q^2-1)$. Then we define a group in terms of a presentation by 
\[
\Lambda_{M,\delta} = \left\langle d,s,a_i  \text{ for all } i \in M \ \left| 
\begin{array}{c}
d^{q+1} = s^2 = (sd)^2 = 1,  \\
s a_i s = a_{qi}, \ d a_i d^{-1} = a_{i+1-q}, \ a_{i+ (q^2 - 1)/2}  a_i = 1 \text{ for all $i \in M$}, \\
a_i a_j = a_{k(i,j)}a_{l(i,j)} \text{ for all $i,j \in M$ with $i \not\equiv j \pmod{q-1}$}
\end{array}
\right.\right\rangle
\]
if $q$ is odd, and if $q$ is even:
\[
\Lambda_{M,\delta} = \left\langle d,s,a_i  \text{ for all } i \in M \ \left| 
\begin{array}{c}
d^{q+1} = s^2 = (sd)^2 = 1,  \\
s a_i s = a_{qi}, \ d a_i d^{-1} = a_{i+1-q}, \ a_i^2 = 1 \text{ for all $i \in M$}, \\
a_i a_j = a_{k(i,j)}a_{l(i,j)} \text{ for all $i,j \in M$ with $i \not\equiv j \pmod{q-1}$}
\end{array}
\right.\right\rangle .
\]

\begin{prop}
\label{prop:explicitLambda}
Let $M,\delta$ be as above. With  the finite set of rational places 
\[
S = \{0,\infty\} \cup \{\delta^{-i(q+1)} \ ; \ i \in M\} \subseteq \bP^1
\]
the assignment $d \mapsto d$, $s \mapsto s$ and $a_i \mapsto [1+\delta^i F]$ defines an isomorphism
\[
\Lambda_{M,\delta} \xrightarrow{\sim} \Lambda_S.
\]
\end{prop}
\begin{proof}
The presentation for $\Lambda_{M,\delta}$ is welldefined because $qi \equiv i \pmod{q-1}$, so all relations use only $a_i$ with $i \in M$. 

Each coset under $(q-1)\bZ/(q^2-1)$ in $M$ contributes to a single $\tau = \delta^{-i(q+1)}$, because the isomorphism $\bZ/(q^2-1) \simeq \bF_{q^2}^\times$, $i \mapsto \delta^i$ translates the norm map $\bF_{q^2}^\times \to \bF_{q}^\times \subseteq \bF_{q^2}^\times$ into multiplication by $q+1$. 

The translation of the relations in the presentation for $\Lambda_S$ from 
Theorem~\ref{thm:structure_Gamma_S} follows essentially from the computation in \eqref{eq:concreterelation1} and \eqref{eq:concreterelation2} above.
\end{proof}

For the sake of completeness we also give the version of the vertex transitive lattices $\Gamma_S$ as a group with an explicit  presentation by 
\[
\Gamma_{M,\delta} = \left\langle a_i  \text{ for all } i \in M \ \left| 
\begin{array}{c}
a_{i+ (q^2 - 1)/2}  a_i = 1 \text{ for all $i \in M$}, \\
a_i a_j = a_{k(i,j)}a_{l(i,j)} \text{ for all $i,j \in M$ with $i \not\equiv j \pmod{q-1}$}
\end{array}
\right.\right\rangle
\]
if $q$ is odd, and if $q$ is even:
\[
\Gamma_{M,\delta} = \left\langle a_i  \text{ for all } i \in M \ \left| 
\begin{array}{c}
a_i^2 = 1 \text{ for all $i \in M$}, \\
a_i a_j = a_{k(i,j)}a_{l(i,j)} \text{ for all $i,j \in M$ with $i \not\equiv j \pmod{q-1}$}
\end{array}
\right.\right\rangle .
\]

\begin{cor}
\label{cor:explicitGamma}
Let $M,\delta$ be as above. With the finite set of rational places 
\[
S = \{0,\infty\} \cup \{\delta^{-i(q+1)} \ ; \ i \in M\} \subseteq \bP^1
\]
the isomorphism of Proposition~\ref{prop:explicitLambda} restricts to an isomorphism,
\[
\Gamma_{M,\delta} \xrightarrow{\sim} \Gamma_S.
\]
\end{cor}

\begin{ex}
\label{ex:666}
We compute the smallest example in dimension $3$ given by $q=5$ and
\[
M = \{i \in \bZ/24\bZ \ ; \ 4 \nmid i\}.
\]
This corresponds to the set $S_0 = \{2,3,4\} \subseteq \bP^1(\bF_5)$, the torsion-free lattice $\Gamma_S$ acts vertex transitively on 
\[
\RT_{6} \times  \RT_{6} \times  \RT_{6},
\]
and  Corollary~\ref{cor:explicitGamma} leads to 
\[
\Gamma_{\{2,3,4\}} = \left\langle \begin{array}{c}
a_1,a_5,a_9,a_{13},a_{17},a_{21}, \\
b_2,b_6,b_{10},b_{14},b_{18},b_{22}, \\
c_3,c_7,c_{11},c_{15},c_{19},c_{23}
\end{array}
\ \left| 
\begin{array}{c}
a_ia_{i+12} = b_i b_{i+12} = c_ic_{i+12} = 1  \ \text{ for all $i$ }, \\

a_{1}b_{2}a_{17}b_{22}, \ 
a_{1}b_{6}a_{9}b_{10}, \ 
a_{1}b_{10}a_{9}b_{6}, \ 
a_{1}b_{14}a_{21}b_{14}, \ 
a_{1}b_{18}a_{5}b_{18}, \\ 
a_{1}b_{22}a_{17}b_{2}, \ 
a_{5}b_{2}a_{21}b_{6}, \ 
a_{5}b_{6}a_{21}b_{2}, \ 
a_{5}b_{22}a_{9}b_{22}, \\

a_{1}c_{3}a_{17}c_{3}, \ 
a_{1}c_{7}a_{13}c_{19}, \ 
a_{1}c_{11}a_{9}c_{11}, \ 
a_{1}c_{15}a_{1}c_{23}, \ 
a_{5}c_{3}a_{5}c_{19}, \\
a_{5}c_{7}a_{21}c_{7}, \ 
a_{5}c_{11}a_{17}c_{23}, \ 
a_{9}c_{3}a_{21}c_{15}, \ 
a_{9}c_{7}a_{9}c_{23}, \\

b_{2}c_{3}b_{18}c_{23}, \ 
b_{2}c_{7}b_{10}c_{11}, \ 
b_{2}c_{11}b_{10}c_{7}, \ 
b_{2}c_{15}b_{22}c_{15}, \ 
b_{2}c_{19}b_{6}c_{19}, \\
b_{2}c_{23}b_{18}c_{3}, \ 
b_{6}c_{3}b_{22}c_{7}, \ 
b_{6}c_{7}b_{22}c_{3}, \ 
b_{6}c_{23}b_{10}c_{23}.
\end{array}
\right.\right\rangle .
\]
\end{ex}

\section{Hurwitz quaternions} 
\label{sec:latticeinhurwitz}
Lattices acting on products of two trees based on rational Hamilton quaternions were
constructed in \cite{mozes:cartan} and \cite{rattaggi:thesis}. Here we extend these constructions to quaternion lattices of higher rank. The notation from Section~\S\ref{sec:lattice} will not apply in this section.

\subsection{The Hurwitz order}
The classical Hamilton quaternion algebra over $\bQ$
\[
D := \left(\frac{-1,-1}{\bQ}\right) = \bQ\{i,j\}/(i^2 = j^2 = -1, ji=-ij).
\]
ramifies exactly at the places $B = \{2, \infty\}$. Let as usual $k := ij$ and 
\[
\rho:=\frac12(1+i+j+k)\in H.
\]
For later use we remark that $\rho^2 = \rho - 1$, i.e., $\rho$ is a $6$-th root of unity.
The Hurwitz maximal order is defined as
\[
\dA := \bZ \oplus \bZ\cdot i \oplus \bZ\cdot j \oplus \bZ\cdot\rho \subseteq D.
\]
Since we consider $\dA$ rather as a coherent sheaf on $\Spec(\bZ)$ of maximal orders of $D$,  we write 
\[
\dA(R) := \dA \otimes_{\bZ} R \qquad \text{ and } \qquad \dA^\times(R) = \big(\dA \otimes_{\bZ} R\big)^\times
\]
for an arbitrary ring $R$. We will also sometimes work with the Lipschitz order
\[
\dA_L := \bZ \oplus \bZ\cdot i \oplus \bZ\cdot j \oplus \bZ\cdot k \subseteq \dA
\]
and set $\dA_L(R):=\dA_L\otimes_{\bZ}R$ for an arbitrary ring $R$. Note that if $2$ is invertible in $R$, then 
\[
\dA(R) = \dA_L(R) = R \oplus R\cdot i \oplus R\cdot j \oplus R\cdot k = R\{i,j\}/(i^2 = j^2 = -1, ji=-ij).
\]
The reduced norm $\Nrd:\dA\to\dO_{\Spec(\bZ)}$ as well as the group schemes $\GL_{1,\dA}$, $\SL_{1,\dA}$ and $\PGL_{1,\dA}$ are defined in a similar way as in Section~\S\ref{sec:lattice}.

\subsection{The $S$-arithmetic group and its volume}
Let $S$ be a finite set of places in $\bZ$ containing $B$. We denote by $\bZ_S$ the localisation of $\bZ$ adding fractions with products of primes from $S$ as denominators. This means in particular that $\bZ_B=\bZ[\frac12]$. We will be concerned with the $S$-arithmetic group
\[
\Lambda_S := \PGL_{1,\dA}(\bZ_S) = \dA^\times(\bZ_S)/\bZ_S^\times = \big(\bZ_S\{i,j\}/(i^2 = j^2 = -1, ji=-ij)\big)^\times/\bZ_{S}^\times
\]
Note that the right equality holds by the same argument as in Section \ref{sec:lattice} since $\bZ_S$ is a principal ideal domain and $2$ is invertible in $\bZ_S$. Furthermore, if $B\subseteq T\subseteq S$, we have an inclusion $\Lambda_T\subseteq \Lambda_S$.

\begin{nota}
As in Section~\S\ref{sec:lattice}, we will use the notation $g\mapsto [g]$ for the map
\[
\GL_{1,\dA}(\bZ_S) \to \PGL_{1,\dA}(\bZ_S)
\]
if it becomes necessary to distinguish between group elements and representatives modulo the center. 
\end{nota}

\begin{nota}
We are interested in the Bruhat-Tits action of $\Lambda_S$. For this purpose, let $S_0 := S\setminus B$ be the set of unramified places in $S$. For each $p\in S_0$, the group $\PGL_{1,\dA}(\bQ_p) \simeq \PGL_2(\bQ_p)$ acts on its Bruhat-Tits building $\RT_{p+1}$, a regular tree of valency $p+1$, with the standard vertex $\star_p$. We set
\[
\PT_{S_0} := \prod_{p\in S_0} \RT_{p+1}
\]
and let $\star_{S_0}\in\PT_{S_0}$ be the product of the distinguished vertices in each factor. 
The arithmetic lattice $\Lambda_S$ acts on $\PT_{S_0}$ through its localisations $\PGL_{1,\dA}(\bQ_p) \simeq \PGL_2(\bQ_p)$ in all $p \in S_0$ as a discrete cocompact subgroup 
\[
\Lambda_S \inj \prod_{p \in S_0} \PGL_2(\bQ_p).
\]
\end{nota}

To compute the orbifold degree $N_0(\Lambda_S)$, we proceed as in Section~\S\ref{sec:lattice} and begin with the exact sequence of (non-abelian) \'etale cohomology
\begin{equation}
\label{eq:comparewith_sc_lattice-Hurwitz}
1 \to \mu_2(\bZ_S) = \{\pm 1\} \to \tilde{\Lambda}_S \to \Lambda_S \xrightarrow{\delta} \rH^1(\bZ_S,\mu_2) = \bZ_S^\times/(\bZ_S^\times)^2,
\end{equation}
where $\tilde{\Lambda}_S=\SL_{1,\dA}(\bZ_S)$. As in Lemma \ref{lem:delta_as_Nrd}, $\delta$ is given by the reduced norm modulo $(\bZ_S^\times)^2$.

\begin{prop}
\label{prop:delta_image-Hurwitz}
For all $S$ with $B \subseteq S$, the image of $\delta: \Lambda_S \to \bZ_{S}^\times/(\bZ_{S}^\times)^2$ is a subgroup of index $2$. 
\end{prop}
\begin{proof}
The reduced norm of every element of $\dA(\bZ_S)^\times$ is positive since $D$ is ramified at $\infty$. Hence it suffices to show that every $p\in S\setminus\{\infty\}$ is a reduced norm of an element of $\dA(\bZ_S)$. But this follows from Jacobi's theorem that every prime number can be written as the sum of four squares of integers.
\end{proof}

\begin{prop}
\label{prop:volumeLambdaS_Hurwitz}
The orbifold degree of the lattice $\Lambda_S$ is
\[
N_0(\Lambda_S) 
=  \frac{1}{24}.
\]
\end{prop}
\begin{proof}
The proof goes similarly to Proposition \ref{prop:volumeLambdaS}. Here the Tits measure $\mu_{\rm Tits}$ can be compared with the Euler-Poincar\'e measure $\mu_{{\rm EP},S_0}$ using \cite[\S4]{borel_prasad} by
\[
N_0(\Lambda_S) \cdot \prod_{p \in S_0} \frac{1-p}{2} = \mu_{{\rm EP},S_0}(\Lambda_S) = \mu_{{\rm Tits}}(\Lambda_S) \cdot \prod_{p \in S_0} \frac{1-p}{1+p}.
\]
Applying Prasad's volume formula \cite[Theorem 3.7]{prasad:volumeformula} yields
\begin{equation*}
\mu_{\rm Tits}(\tilde{\Lambda}_S) = \frac{1}{4\pi^2} \cdot \zeta_\bQ(2) \cdot \prod_{p \in B\setminus\{\infty\}} (p-1) \cdot \prod_{p \in S_0} (p + 1)  = \frac{1}{24} \cdot \prod_{p \in S_0} (p + 1).
\end{equation*}
Hence by Proposition~\ref{prop:delta_image-Hurwitz}, \eqref{eq:comparewith_sc_lattice-Hurwitz} and Dirichlet's $S$-Unit Theorem,
\begin{align*}
\mu_{\rm Tits}(\Lambda_S) = \mu_{\rm Tits}(\tilde \Lambda_S) \cdot \frac{\# \mu_2(\bZ_S)}{\frac{1}{2}\# \bZ_S^\times/(\bZ_S^\times)^2} = \frac{\mu_{\rm Tits}(\tilde \Lambda_S)}{2^{\#S_0}} = \frac{1}{24} \cdot \prod_{p \in S_0} \frac{p + 1}{2}.
\end{align*}
The claimed formula for $N_0(\Lambda_S)$ follows at once.
\end{proof}

\subsection{The $2$-adic filtration}
We now work locally at the ramified place given by $2$. The base extension to the field of $2$-adic numbers
\[
\cD_2:=D\otimes_\bQ\bQ_2 = \left(\frac{-1,-1}{\bQ_2}\right)
\]
is a skew field and has thus a unique maximal order $\fO_{D,2}\subseteq \cD_2$. Its unique maximal ideal $\fm_2\unlhd\fO_{D,2}$ is generated by $\pi:=i-j$, an element of reduced norm $2$ with $\pi^2=-2$. Furthermore, the extension $\bQ_2(\rho)/\bQ_2$ is unramified (recall that $\rho^2  = \rho-1$) and $\pi(-)\pi^{-1}$ induces the nontrivial Galois automorphism $\sigma$ because
\[
\pi\rho = -j+k = (i-j)-(i-k) = \pi-\rho\pi = (1-\rho)\pi.
\]
Let $\bQ_2(\rho)\{\pi\}$ be the $\sigma$ semi-linear polynomial ring, i.e., $\sigma(x) = \pi x \pi^{-1}$ for all $x \in \bQ_2(\rho)$. Then we obtain
\begin{align*}
\cD_2 & = \bQ_2(\rho)\{\pi\}/(\pi^2=-2) \\ 
\fO_{D,2} & = \bZ_2[\rho]\{\pi\}/(\pi^2=-2)  \subseteq \cD_2.  
\end{align*}
For each $n\in\bN$, we define the $n$-th higher unit group as
\[
U^n = U^n_{D,2} := 1+\fm_2^n \subseteq \fO_{D,2}^\times.
\]

\begin{lem}
Let $i,r,s\in\cD_2^\times/(\bQ_2^\times\cdot U^2)$ be the images of $i$, $\rho$ and $ij\pi = i+j$ respectively.
\begin{enumerate}
\item
The group $\fO_{D,2}^\times/U^2$ is isomorphic to the alternating group $A_4$ and has the following presentation:
\begin{equation}\label{eq:OD2}
\fO_{D,2}^\times/U^2 = \langle i,r \ | \ i^2, r^3, (ir)^3 \rangle.
\end{equation}
\item
The group $\cD_2^\times/(\bQ_2^\times\cdot U^2)$ is isomorphic to the symmetric group $S_4$ and has the following presentation:
\begin{equation}\label{eq:D2modU2}
\cD_2^\times/(\bQ_2^\times\cdot U^2) = \langle r,s \ | \ r^3, s^2, (rs)^4 \rangle.
\end{equation}
\end{enumerate}
\end{lem}
\begin{proof}
(1) Observe first that $\fO_{D,2}^\times/U^2$ sits in the exact sequence
\[
1 \to U^1/U^2 \to \fO_{D,2}^\times/U^2 \to \fO_{D,2}^\times/U^1 \to 1.
\]
Since $\fO_{D,2}/\fm_2\simeq\bF_4$, it is easy to see that $U^1/U^2 \simeq (\bF_4,+) \simeq V_4$, the Klein four-group with generators $i,j$, and $\fO_{D,2}^\times/U^1\simeq\bF_4^\times$, a cyclic group of order $3$ with generator $\rho=[r]$. Hence we obtain the presentation
\[
\fO_{D,2}^\times/U^2 = \langle i,j,r \ | \ i^2,j^2,(ij)^2,r^3,rir^{-1}=j,rjr^{-1}=ij \rangle.
\]
After substituting $j$ by $rir^{-1}$, we get the presentation of $\fO_{D,2}^\times/U^2$ as in \eqref{eq:OD2}

(2) Let $v_2:\cD_2^\times\to\bZ$ be the valuation on $\cD_2$. This is surjective and yields the exact sequence
\[
1 \to \fO_{D,2}^\times/U^2 \to \cD_2^\times/(\bQ_2^\times\cdot U^2) \xrightarrow{v_2\bmod2} \bZ/2\bZ \to 1
\]
with a section $\bZ/2\bZ \to \cD_2^\times/(\bQ_2^\times\cdot U^2)$ given by $1\mapsto s=[i+j]$. Since $sis^{-1}=j$ and $srs^{-1}=ir^{-1}$, we obtain the following presentation:
\[
\cD_2^\times/(\bQ_2^\times\cdot U^2) = \langle i,r,s \ | \ i^2,r^3,(ir)^3,s^2,sis^{-1}=rir^{-1},srs^{-1}=ir^{-1}\rangle.
\]
The last relation together with $s^2=1$ implies that $i=srsr$. Substituting this in the above presentation, we obtain \eqref{eq:D2modU2} as desired.
\end{proof}

We can also use the $2$-adic filtration to describe the group $\Lambda_B$ as follows:

\begin{prop} \label{prop:LambdaBHurwitz}
The restrictions of the group homomorphism
\[
\psi' : D^\times \subseteq \cD_2^\times \twoheadrightarrow \cD_2^\times/(\bQ_2^\times\cdot U^2)
\]
yields group isomorphisms
\[
\Lambda_B \simeq \cD_2^\times/(\bQ_2^\times\cdot U^2) \quad \text{and} \quad \dA^\times/\{\pm1\} \simeq \fO_{D,2}^\times/U^2
\]
\end{prop}
\begin{proof}
Consider a unit $w=u+vi+xj+yk \in \dA(\bZ_B)^\times$. After scaling by a power of $2$, we may assume that $u,v,x,y$ are all integers in $\bZ$ and at least one of them is odd. Since $w \in \dA(\bZ_B)^\times$, its reduced norm $\Nrd(w)=u^2+v^2+x^2+y^2$ must be a power of $2$. Since the square of an odd number is congruent to 1 modulo 8, $w$ is invertible in $\dA(\bZ_B)$ if and only if $\Nrd(w)\in\{1,2,4\}$. Hence $\Lambda_B$ can be represented by the following elements:
\begin{equation} \label{eq:listLambdaB}
1, i, j, k, 1\pm i, 1\pm j, 1\pm k, i\pm j, i\pm k, j\pm k, 1\pm i\pm j\pm k.
\end{equation}
Among these elements, only $1$ lies in $(\bQ_2^\times\cdot U^2)$. Hence the induced map $\Lambda_B \to \cD_2^\times/(\bQ_2^\times\cdot U^2)$ is injective, thus also bijective since both $\Lambda_B$ and $\cD_2^\times/(\bQ_2^\times\cdot U^2)$ have exactly $24$ elements.

Since the elements of $\dA^\times/\{\pm1\}$ as subgroup of $\Lambda_B$ are represented by those from \eqref{eq:listLambdaB} with trivial $2$-adic valuation modulo $2$, we also get the isomorphism $\dA^\times/\{\pm1\} \simeq \fO_{D,2}^\times/U^2$ as desired.
\end{proof}
\begin{nota}
The composition of $\psi'$ from Proposition \ref{prop:LambdaBHurwitz} with the isomorphism $\cD_2^\times/(\bQ_2^\times\cdot U^2)\simeq\Lambda_B$ will be denoted by
\[
\psi : D^\times \to \Lambda_B.
\]
Since $\psi$ factorises over $D^\times/K^\times$, we will also, by abuse of notation, denote by $\psi$ the induced homomorphism $\Lambda_S\subseteq D^\times/K^\times\to\Lambda_B$ for any finite set $S$ of places in $\bZ$ containing $B = \{2,\infty\}$. Note that this is surjective since the restriction to $\Lambda_B\subseteq\Lambda_S$ is the identity.
\end{nota}

\subsection{The generator sets $A_p$}
Let $p$ be an odd prime and
\[
A_p := \{ x\in\dA(\bZ) \ ; \ \Nrd(x)=p \text{ and } \psi(x)=1\}.
\]
Note that each element of $A_p$ lies in the Lipschitz oder $\dA_L$ and is modulo $2\dA_L$ congruent to $1$ if $p\equiv1\bmod4$ and $i+j+k$ if $p\equiv3\bmod4$. 
Furthermore, let $PA_p$ be the image of $A_p$ under the quotient map in $\Lambda_p:=\Lambda_{\{0,\infty,p\}}$. We are going to show that the action of $PA_p$ on the standard vertex $\star_p\in\RT_{p+1}$ yields a bijection between $PA_p$ and the neighbours of $\star_p$. For this we need the following lemma:

\begin{lem}\label{lem:stabstdvertexinHurwitz}
Let $S$ be a finite set of places of $\bZ$ containing $B = \{2,\infty\}$ and $S_0 = S \setminus B$. The stabiliser of $\star_{S_0}\in\PT_{S_0}$ for the action of $\Lambda_S$ is $\Lambda_B$.
\end{lem}
\begin{proof}
This is done in the same way as the proof of Proposition \ref{prop:GammaS} \ref{propitem:GammaS1}.
\end{proof}

\begin{prop}\label{prop:reducednorm_p_action}
Let $p$ be an odd prime. The action of $PA_p$ on the standard vertex $\star_p\in\RT_{p+1}$ yields a bijection between $PA_p$ and the neighbours of $\star_p$.
\end{prop}

\begin{proof}
Observe first that the right ideals of reduced norm $p$ in $\dA$ are in a canonical bijection to the proper right ideals of $\dA/p\dA \simeq\rM_2(\bF_p)$, which are given by elements of $\bF_p^2\setminus\{0\}$ up to scalar multiplication. Hence there are exactly $p+1$ such ideals. Each of them is principal, i.e.~generated by an element in $\dA$ of reduced norm $p$. Such an element lies in $\fO_{D,2}^\times$ and can hence be multiplied from the right with some $u\in\dA^\times$ to get an $x\in\dA$ with reduced norm $p$ and $\psi(x)=1$, which generates the same right ideal. Here $u$ is unique up to sign, which shows that $PA_p$ has exactly $p+1$ elements.

To see that $PA_p$ maps $\star_p$ bijectively to all its neighbours, suppose first that $x_1,x_2\in PA_p$ are such that $x_1\star_p=x_2\star_p$. Then $x_1^{-1}x_2$ stabilises the standard vertex, thus lies in $\Lambda_B$ by Lemma \ref{lem:stabstdvertexinHurwitz}. On the other hand, since $\psi(x_1^{-1}x_2)=1$, we have $x_1^{-1}x_2=1$ by Proposition \ref{prop:LambdaBHurwitz}, hence $x_1=x_2$. Thus the claim follows by the cardinality argument.
\end{proof}

\begin{prop}\label{prop:localpermutationHurwitz}
Let $p,\ell$ be two different odd primes. There are unique permutations $\sigma^\ell_x\in\Aut(PA_\ell)$ for each $x\in PA_p$ and $\sigma^p_y\in\Aut(PA_p)$ for each $y\in PA_\ell$ such that
\[
x\cdot\sigma^\ell_x(y) = y\cdot\sigma^p_y(x)
\]
for all $x\in PA_p$ and $y\in PA_\ell$.
\end{prop}
\begin{proof}
Since $PA_p$ and $PA_\ell$ are stable under inversion, this is equivalent to show that for each $x\in PA_p$ and $y\in PA_\ell$, there are uniquely determined $x'\in PA_p$ and $y'\in PA_\ell$ such that $xy=y'x'$. This follows from properties of unique factorization valid for $\dA$ proven by Hurwitz to be recalled below. The same results were used in \cite[Cor 3.11(1)]{rattaggi:thesis}.

Let $x_0,y_0 \in \dA$ be representatives of $x$ and $y$ of reduced norm $p$ and $\ell$ respectively. Since $x_0y_0$ has square-free reduced norm $p\ell$, it follows from unique factorization in the Hurwitz order as in \cite[II \S5]{conway-smith}  that there is a $y'_0\in\dA$ with reduced norm $\ell$ such that
\begin{equation}\label{eq:VHstructureHurwitz}
x_0y_0=y'_0x'_0
\end{equation}
for some $x'_0\in\dA$. Here $y'_0$ is unique up to right multiplication with a unit, 
see  \cite[II \S5]{conway-smith}. Hence by Proposition \ref{prop:LambdaBHurwitz}, there is a unique way to choose $y'_0$ such that $\psi(y'_0)=1$. Consequently, there is also only one $x'_0\in\dA$ satisfying \eqref{eq:VHstructureHurwitz}. Applying the reduced norm and also $\psi$  shows that $x'_0$ has reduced norm $p$ and $\psi(x'_0)=1$ as desired.
\end{proof}

\subsection{Presentation of $\Lambda_S$}
As in the function field case, we now describe $\Lambda_S$ in terms of generators and relations.

\begin{prop}
\label{prop:GammaS-Hurwitz}
Let $S$ be a finite set of places of $\bQ$ containing $B = \{2,\infty\}$ and $S_0 = S \setminus B$.

\begin{enumerate}
\item 
\label{propitem:GammaS-Hurwitz1}
The group
\[
\Gamma_S = \langle PA_p \ ; \ p\in S_0\rangle.
\]
acts transitively on the vertices of $\PT_{S_0}$
\item
The group $\Lambda_S$ is generated by $\Lambda_B$ and all $PA_p$ for $p\in S_0$. More precisely,
\[
\Lambda_S = \Gamma_S\Lambda_B.
\]
\item
The homomorphism $\psi:\Lambda_S\to\Lambda_B$ fits into the exact sequence
\[
1 \to \Gamma_S \to \Lambda_S \xrightarrow{\psi} \Lambda_B \to 1.
\]
\item $\Gamma_S$ is a normal subgroup of $\Lambda_S$ of index $24$.
\end{enumerate}
\end{prop}

\begin{proof}
By Proposition \ref{prop:reducednorm_p_action}, the standard vertex $\star_{S_0}$ is moved by elements of $PA_p$ to all neighbours in $p$-direction. Hence all neighbours of $\star_{S_0}$ lie in its $\Gamma_S$-orbit, which proves (1). Assertion (2) follows from (1) and Lemma \ref{lem:stabstdvertexinHurwitz}. Also (3) follows from (2) and (4) follows from (3).
\end{proof}

\begin{thm}
\label{thm:structure_Gamma_S-Hurwitz}
Let $S$ be a finite set of places of $\bQ$ containing $B = \{2,\infty\}$ and $S_0 = S \setminus B$. 
\begin{enumerate}
\item 
The group $\Gamma_S$ has the following presentation:
\begin{itemize}
\item
$\Gamma_S$ is generated by $\bigcup_{p \in S_0} PA_p$,
\item
subject to the following relations:
\begin{enumerate}
\item[(i)]
for all $p \in S_0$ and $[x_0 + x_1i + x_2j + x_3k] \in PA_p$, we have
\[
[x_0 + x_1i + x_2j + x_3k] \cdot [x_0 - x_1i - x_2j - x_3k] = 1,
\]
\item[(ii)]
for all $p \not= \ell \in S_0$ and $x \in PA_p$, and $y \in PA_\ell$, we have
\[
x\cdot\sigma^\ell_x(y) = y\cdot\sigma^p_y(x).
\]
\end{enumerate}
\end{itemize}
Here $\sigma^\ell_x$ and $\sigma^p_y$ the permutations determined by 
Proposition~\ref{prop:localpermutationHurwitz}.
\item
If $p\equiv1\bmod4$ for all $p\in S_0$ or $p\equiv\pm1\bmod8$ for all $p\in S_0$, the group $\Gamma_S$ acts on the cube complex $\PT_{S_0}$ without fixing a cell. In particular, $\Gamma_S$ is torsion-free.
\end{enumerate}
\end{thm}

Note that this generalises the $2$-dimension results of \cite[\textsection3]{mozes:cartan} and \cite[Cor.3.11, Thm.3.30]{rattaggi:thesis}.

\begin{proof}
(1) is done in the same way as Theorem \ref{thm:structure_Gamma_S}. Also by the argument in loc.~cit., it suffices to show that $\Gamma_S$ has no $2$-torsion in the given cases to prove (2). In fact, each element of $\Gamma_S\setminus\{1\}$ has a representative $x\in\dA_L$ which is a product of elements of $A_p$ for $p\in S_0$. If $p\equiv1\bmod4$ for all $p\in S_0$, we have $x\equiv1\bmod2\dA_L$. If $p\equiv\pm1\bmod8$, we have $x\equiv1\bmod2$ or $x\equiv i+j+k\bmod2$ but has reduced norm congruent to $7$ modulo $8$, hence not of the form $ai+bj+ck$ with $a,b,c\in\bZ$. In either case, the reduced trace of $x$ does not vanish, implying that $[x]\in\Gamma_S$ is not a $2$-torsion element.
\end{proof}

\begin{cor} \label{cor:structure_Gamma_S-Hurwitz}
Let $S$ be a finite set of places of $\bZ$ containing $B = \{2,\infty\}$ and $S_0 = S \setminus B$. 
The group $\Lambda_S$ is a semi-direct product $\Lambda_S = \Gamma_S \rtimes \Lambda_B$, and we have the following presentation:
\begin{itemize}
\item
$\Lambda_S$ is generated by $r, s$ and $PA_p$ for each $p \in S_0$,
\item
for all $p \in S_0$ and $x=[x_0+x_1i+x_2j+x_3k]\in PA_p$, we have
\begin{enumerate}
\item[(i)] $[x_0+x_1i+x_2j+x_3k]\cdot[x_0-x_1i-x_2j-x_3k]=1$,
\item[(ii)] $r\cdot[x_0+x_1i+x_2j+x_3k]\cdot r^{-1} = [x_0+x_3i+x_1j+x_2k]$,
\item[(iii)] $s\cdot[x_0+x_1i+x_2j+x_3k]\cdot s^{-1} = [x_0+x_2i+x_1j-x_3k]$,
\end{enumerate}
\item
for all $p \not= \ell \in S_0$ and $x \in PA_p$, and $y \in PA_\ell$ we have
\[
x\cdot\sigma^\ell_p(y) = y\cdot\sigma^p_y(x).
\]
\end{itemize}
\end{cor}
\begin{proof}
This is done in the same way as Corollary \ref{cor:structure_Gamma_S}
\end{proof}

\subsection{A torsion-free simply transitive lattice}
Although Theorem \ref{thm:structure_Gamma_S-Hurwitz} does not always yield a torsion-free lattice, we can still find a torsion-free arithmetic subgroup of $\Lambda_S$ acting on the vertices of $\PT_{S_0}$ simply transitively. For this we follow the strategy in \cite[\textsection3.3]{rattaggi:thesis} and define
\[
A'_p := \begin{cases} A_p & \text{if} ~ p\equiv 1\bmod4 \\ A_pi & \text{if} ~ p\equiv 3\bmod4, \end{cases}
\]
and let $PA'_p\subseteq\Lambda_p$ be the image of $A'_p$ under the quotient map.

\begin{prop}
\label{prop:torsionfreeinSHurwitz}
Let $p$ and $\ell$ be different odd primes.
\begin{enumerate}
\item
The action of $PA'_p$ on the standard vertex $\star_p\in\RT_{p+1}$ yields a bijection between $PA'_p$ and the neighbours of $\star_p$.
\item
There are permutations $\tau^\ell_x\in\Aut(PA'_\ell)$ for each $x\in PA'_p$ and $\tau^p_y\in\Aut(PA'_p)$ for each $y\in PA'_\ell$ such that
\[
x\cdot\tau^\ell_x(y) = y\cdot\tau^p_y(x)
\]
for all $x\in PA'_p$ and $y\in PA'_\ell$.
\end{enumerate}
\end{prop}
\begin{proof}
These can be done in the same way as Propositions \ref{prop:reducednorm_p_action} and \ref{prop:localpermutationHurwitz}.
\end{proof}

\begin{thm}
\label{thm:torsionfreelattice_S}
Let $S$ be a finite set of places of $\bZ$ containing $B = \{2,\infty\}$ and $S_0 = S \setminus B$. Define
\[
\Gamma'_S := \langle PA'_p \ | \ p\in S_0 \rangle.
\]
\begin{enumerate}
\item
Every element in $\Gamma'_S$ has a representative in $\dA_L$ congruent to $1$ or $1+j+k$ modulo $2\dA_L$.
\item
The action of $\Gamma'_S$ on $\PT_{S_0}$ is simply transitive on the vertices and has no fixed cells. In particular, $\Gamma'_S$ is torsion-free.
\end{enumerate}
\end{thm}
\begin{proof}
Observe that each element in $A'_p$ as element of $\dA_L$ is congruent to $1$ or $1+j+k$ modulo $2\dA_L$. This proves (1) since $\{1,1+j+k\}$ is closed under the multiplication in $\dA/2\dA_L$. To prove (2), it suffices by the same argument as in Theorem \ref{thm:structure_Gamma_S-Hurwitz} to show that $\Gamma'_S$ has no $2$-torsion. This holds since each element of $\Gamma'_S$ has a representative in $\dA_L$ which is congruent to $1$ or $1+j+k$ modulo $2\dA_L$ and thus has a reduced trace $\equiv 2 \pmod 4$. Note that elements representing of order $2$ must have vanishing reduced trace.
\end{proof}

\begin{ex}
\label{ex:468}
We compute the smallest example in dimension $3$ given by $S_0 = \{3,5,7\}$. The torsion-free group $\Gamma'_S$ acts vertex transitively on the product
\[
 \RT_{4} \times  \RT_{6} \times  \RT_{8}.
\]
Here the generator sets are in terms of representatives in $\dA^\times (\bZ[\frac{1}{210}])$:
\begin{align*}
PA'_3 & = \{1 \pm i \pm j\}, \\
PA'_5 & = \{1 \pm 2i, 1 \pm 2j, 1 \pm 2k\}, \\
PA'_7 & = \{1 \pm 2i \pm j \pm k\}.
\end{align*}
The relations are determined by Proposition~\ref{prop:torsionfreeinSHurwitz}. More concretely, we set
\begin{align*}
PA'_3: \qquad & a_1 = 1 + j + k, \ a_2 = 1+j-k, \ a_3 = 1-j-k, \ a_4 = 1 -j + k, \\
PA'_5: \qquad  & b_1 = 1 + 2i, \ b_2  = 1 + 2j, \ b_3 = 1 + 2k, \ b_4 = 1 - 2i, \ b_5 = 1- 2j, \ b_6 = 1 - 2k, \\
PA'_7: \qquad  & c_1 = 1+2i + j + k, \ c_2 = 1-2i + j + k, \ c_3 = 1+2i - j + k, \ c_ 4= 1+2i + j - k, \\ &  c_5 = 1- 2i -  j - k, \ c_6 = 1+2i - j - k, \ c_7 = 1-2i + j - k, \ c_8 = 1-2i - j + k.
\end{align*}
With this notation we have $a_i^{-1} = a_{i+2}$, $b_i^{-1} = b_{i+3}$, and $c_i^{-1} = c_{i+4}$, and using these abbreviations we find the explicit presentation
\[
\Gamma'_{\{3,5,7\}} = \left\langle
\begin{array}{c}
a_1,a_2 \\ 
b_1,b_2,b_3 \\
c_1,c_2,c_3,c_4
\end{array}
\ \left| 
\begin{array}{c}
a_1b_1a_4b_2,  \ a_1b_2a_4b_4, \  a_1b_3a_2b_1, \ 
a_1b_4a_2b_3,  \ a_1b_5a_1b_6, \ a_2b_2a_2b_6 \\

a_1c_1a_2c_8, \ a_1c_2a_4c_4, \ a_1c_3a_2c_2, \ a_1c_4a_3c_3, \\
a_1c_5a_1c_6, \ a_1c_7a_4c_1, \ a_2c_1a_4c_6, \ a_2c_4a_2c_7 \\

b_1c_1b_5c_4, \
b_1c_2b_1c_5, \
b_1c_3b_6c_1, \
b_1c_4b_3c_6, \
b_1c_6b_2c_3, \
b_1c_7b_1c_8, \\
b_2c_1b_3c_2, \
b_2c_2b_5c_5, \
b_2c_4b_5c_3, \
b_2c_7b_6c_4, \
b_3c_1b_6c_6, \
b_3c_4b_6c_3
\end{array}
\right.\right\rangle.
\]
\end{ex}

\section{Combinatorial cube complexes}

Cube complexes and in particular CAT(0) cube complexes feature prominently in geometric group theory. Cube complexes are usually defined as cellular topological spaces consisting of cubes with attaching maps respecting the cubical structure. For an introduction we refer to 
\cite{sageev:cubecomplexes} and  \cite{ballmann_swiatkowski}. There is however an alternative categorical definition similar to simplicial complexes. This combinatorial description is more suitable for our treatment of the doubling construction in higher rank, and potentially for future applications.

\subsection{The cube category, cubical sets and cube complexes}
We omit degeneracy maps for simplicity of exposition and because our modest applications of the formalism does not require them. Our category of cubes $\Cube$ rather is an analogue of the category  underlying semi-simplicial sets.

We refer to \cite{grandis_mauri} for a more in depth categorical analysis of cubical structures. Square complexes have been treated similarly in \cite[\S1]{burger-mozes:lattices}.

\begin{defi}
\label{defi:cube}
The \textbf{cube category} 
$\Cube$ has the following objects and morphisms. 
\begin{enumerate}
\item[(i)]
Objects $[n]$ for each $n \in \bN_{\geq 0}$. 
\item[(ii)]
Morphisms $f: [n] \to [m]$ in $\Cube$ are maps $f: \{-1,1\}^n \to \{-1,1\}^m$ of the form 
\[
f (a_1, \ldots, a_n) = (\ep_1 a_{\sigma^{-1}(1)}, \ldots, \ep_m a_{\sigma^{-1}(m)}),
\]
where $\sigma: \{1,\ldots,n\}  \inj \{1,\ldots,m\}$ is an injective map, $\ep$ is a map of signs 
\[
\ep :  \{1,\ldots, m\}  \to \{-1,1\}
\]
and, by abuse of notation, we write $a_{\sigma^{-1}(j)} :=  1$, $j \notin \sigma(\{1,\ldots,n\})$.
Composition of morphisms in $\Cube$ is composition of maps.
\end{enumerate}
A \textbf{cubical set} is a contravariant functor $X : \Cube^\op \to \sets$. Morphisms of cubical sets are natural transformations of functors. We abbreviate the effect of $X$ on morphisms $f : [n] \to [m]$ by $f^\ast := X(f)$. 
\end{defi}

\begin{rmk} 
The map $f$ assigned to $\sigma$ and $\ep$ as in the definition of $\Cube$ is the map that moves $a_i$ to the $j=\sigma(i)$-th position and multiplies it with $\ep_j$, and inserts signs $\ep_j$ at the $j$th place for $j$ not in the image of $\sigma$. 
\end{rmk}

\begin{rmk}
All morphisms in $\Cube$ are composites of automorphisms and products of \textbf{face maps} 
\[
\delta^\ep_i : [n]  \to [n+1],
\]
the map that inserts $\ep \in \{-1,1\}$ into the $i$th coordinate, for $1 \leq i \leq n+1$, i.e., 
\[
\delta^\ep_i(a_1,\ldots,a_n) = (a_1,\ldots,a_{i-1},\ep,a_i, \ldots, a_{n}).
\]
The automorphism group of the $n$-dimensional cube $[n] \in \Cube$ is
\[
G_n := \Aut_{\Cube}([n]) = \{\pm 1\} \wr S_n = \{\text{signed permutation matrices in } \GL_n(\bZ)\}.
\]
\end{rmk}

\begin{rmk}
Let $X : \Cube^\op \to \sets$ be a cubical set. The \textbf{parametrized cubes} of dimension $n$ is the set with right $G_n$-action 
\[
X_n := X([n]).
\]
A \textbf{cube} of dimension $n$ of $X$ is a $G_n$-orbit in $X_n$.  
For a parametrized cube $x \in X_n$ the parametrized cubes of dimension $n-1$
\[
\delta_i^{\ep,\ast}(x)
\]
are the \textbf{faces} of $x$ in position $\ep$ with respect to the $i$th coordinate. 
\end{rmk}

\begin{defi}
Let $I = [-1,1]$ be the real interval. The group $G_n$ acts geometrically on $I^n$ in a natural way by signed permutation matrices. Similarly, for each map $f: [n] \to [m]$ in $\Cube$ we have the corresponding linear map
\[
f_I : I^n \to I^m.
\]

The \textbf{geometric realization} of a cubical set $X$ is the topological space
\[
\abs{X} = \coprod_{n} X_n \times I^n/\sim
\]
where $\sim$ is generated by $\big(f^\ast(x), t \big) \sim  \big(x, f_I(t)\big)$
for all $f: [n] \to [m]$, all $x \in X_m$ and all $t \in I^{n}$. 

Note that the $G_n$-orbit of a cube that corresponds to changes of orientation of a fixed parametrized cube $x \in X_n$ gives rise to a single topological $n$-cell.
\end{defi}

\begin{defi}
In this note, a \textbf{cube complex (of dimension $n$)} is defined to be a cubical set $X$ with the following properties.
\begin{enumerate}
\item[(i)]
$X$ is of pure dimension $n$, i.e., for all $m > n$ we have $X_m = \emptyset$ and all $y \in X_{m}$ for $m <n$ are faces $y= \delta_i^{\ep,\ast}(x)$ of some $x \in X_{m+1}$.
\item[(ii)]
For all $m \leq n$ the $G_m$-action on $X_m$ is free.
\end{enumerate}
\end{defi}

\begin{rmk}
Traditionally, cube complexes of dimension $n$ are cellular complexes $\cX$ such that 
\begin{enumerate}
\item[(i)] each cell is isomorphic to $I^m$ for some $m \leq n$, and 
\item[(ii)] attaching maps are maps that factor via the geometric realization of face maps over attaching maps of cubes of smaller dimension. 
\end{enumerate}
If we consider such a topological space $\cX$ as endowed with cubical structure such that 
\[
X_n = \Hom(I^n,\cX)
\]
consists of a free right $G_n$-set for each geometric $n$-dimensional cube, then the $X_n$ naturally enrich to a functor $X : \Cube^\op \to \sets$ which is a cube complex in our sense. Its topological realization yields naturally $\abs{X} = \cX$ by evaluation $\Hom(I^n,\cX) \times I^n \to \cX$.
\end{rmk}

\begin{ex}
The cubical set $\square^n = \Hom(-,[n])$ is called the $n$-dimensional cube, and for a good reason. Its topological realisation is easily seen to be
\[
\abs{\square^n} = I^n,
\]
the $n$-dimensional standard cube on which we model topological realizations of cube complexes. The $p$-dimensional faces of $\square^n$ are 
\[
(\square^n)_p = \Hom([p],[n]) = \{f: \{-1,1\}^p \to \{-1,1\}^n \ ; \  \text{ as in Definition~\ref{defi:cube}} \}
\]
of which there are 
\[
\#(\square^n)_p = p! \binom{n}{p} \cdot 2^n.
\]
\end{ex}

\subsection{Product decomposition of the standard cube}

For a group $G$ we denote the category of right $G$-sets by $\sets_G$.
Let $H$ be a subgroup of $G$. Restriction $\res_H^G: \sets_G \to \sets_H$  has a left adjoint 
\[
M  \times^H G = (M \times G)/\sim,
\]
where $M$ is a right $H$-set and $\sim$ is defined by 
\[
(m,hg) \sim (mh,g) \qquad \text{ for all } m \in M, g \in G, h \in H.
\]
The right $G$-action on $M \times^H G$ is by right translation on the second factor.
The natural bijection establishing adjointness is
\begin{align*}
\Hom_G(M \times^H G ,A) & = \Hom_H(M,\res_{H}^G(A)) \\
\big(\ph: M \times^H G \to A \big) & \mapsto  \ph(-,1) \\
\big((m,g) \mapsto \psi(m)g \big) & \mapsfrom \big(\psi: M \to A\big).
\end{align*}

\begin{defi}
The cube category carries a \textbf{monoidal structure} $\times$ as follows. For objects we set
\[
[n] \times [m] = [n+m]
\]
and for morphisms $\ph: [p] \to [a]$ and $\psi: [q] \to [b]$ we set
\[
\ph \times \psi : [p+q] \to [a + b], \qquad \ph \times \psi(t_1, \ldots,t_{p+q}) = \ph(t_1,\ldots,t_p) \oplus \psi(t_{p+1}, \ldots, t_{p+q}),
\]
where the sum $\oplus$ of tuples of values of $\ph$ and $\psi$ is meant in the sense of concatenation of tuples. 
\end{defi}

For $p+q = n$ the group $G_p \times G_q$ is naturally a subgroup of $G_n$ via $\times$. Moreover, the $\times$-product on maps together with the adjointness to restriction induces a map
\begin{equation}
\label{eq:decompose_cube}
\coprod_{p+q = n}   \big((\square^{a})_p \times (\square^{b})_q\big) \times^{G_p \times G_q} G_n
\to (\square^{a+b})_n.
\end{equation}
This map is clearly surjective and hence bijective as the following comparison of cardinalities shows:
\[
\sum_{p+q = n} 2^a p! \binom{a}{p} \cdot 2^b q! \binom{b}{q} \cdot \frac{2^n n!}{2^p p! \cdot 2^q q!} = 2^{a+b} n! \sum_{p+q = n}  \binom{a}{p}  \binom{b}{q} = 2^{a+b} n! \binom{a+b}{n}.
\]

\subsection{Products of cube complexes} There are several notions of products of cube complexes.

\begin{defi}
The \textbf{cartesian product} of cubical sets is the product in the sense of presheaves of sets on the category $\Cube$. For cubical sets $X,Y$ it is the degreewise product $X \boxtimes Y : \Cube^\op \to \sets$ 
\begin{align*}
(X \boxtimes Y)_n &= X_n \times Y_n, \\
X \boxtimes Y(f) & =  X(f) \times Y(f).
\end{align*}
Being a product there are natural projections
\[
\pr_X: X \boxtimes Y \to X, \qquad \pr_Y : X \boxtimes Y \to Y.
\]
\end{defi}

\begin{defi}
More generally, there are \textbf{cartesian fibre products} of cubical sets  in the sense of presheaves of sets on the category $\Cube$. For cubical sets $X,Y,S$ and maps $X \to S$ and $Y \to S$ it is the degreewise fibre product $X \boxtimes_S Y : \Cube^\op \to \sets$ defined by 
\begin{align*}
(X \boxtimes_S Y)_n &= X_n \times_{S_n} Y_n, \\
X \boxtimes_S Y(f) & =  X(f) \times_{S(f)} Y(f).
\end{align*}
\end{defi}

Cube complexes are combinatorial blueprints for cell complexes whose cells are squares and all glueings are performed combinatorially. Hence, it comes at no surprise that there is a product of cube complexes that describes the product in topological spaces for their topological realization.

\begin{defi}
The \textbf{(topological) product} of cubical sets $X,Y$ it is defined as  $X \boxtimes Y : \Cube^\op \to \sets$ by  
\[
(X \times Y)_n = \coprod_{p+q = n} \big(X_p \times Y_q\big) \times^{G_p \times G_q} G_n.
\]
For the effect on maps  we have to define for every $f \in [n] \to [m]$, i.e., $f \in (\square^m)_n$ a map 
\[
f^\ast : \coprod_{a+b = m} \big(X_a \times Y_b\big) \times^{G_a \times G_b} G_m  \to \coprod_{p+q = n} \big(X_p \times Y_q\big) \times^{G_p \times G_q} G_n.\
\]
For a pair of cubes $x \in X_a$ and $y \in Y_b$ together with some $g \in G_m$ we define the effect of $f^\ast$ using that according to \eqref{eq:decompose_cube} 
\[
gf = (\ph \times \psi) h
\]
for some $\ph : [p] \to [a]$ and $\psi : [q] \to [b]$ where $p+q = n$ and some $h \in G_n$. Then
\[
f^\ast((x,y)g) = (\ph^\ast(x),\psi^\ast(y))h. 
\]
\end{defi}

We omit the tedious task of checking that $X \times Y$ is indeed a well defined functor $\Cube^\op \to \sets$. 

\begin{ex}
\label{ex:productofcubes}
For all $a,b \geq 0$ and $m=a+b$ the maps \eqref{eq:decompose_cube} yield a natural isomorphism
\[
\square^a \times \square^b \xrightarrow{\sim}  \square^{m}.
\]
\end{ex}

\begin{prop}
The topological product corresponds to the product in topological spaces under the functor geometric realisation. For cubical sets $X$ and $Y$ we have natural homeomorphisms
\[
\abs{X \times Y} = \abs{X} \times \abs{Y}.
\]
\end{prop}
\begin{proof}
A cubical set $X$ is the result of combinatorially gluing its cubes $x : \square^n \to X$. Similarly, its realization $\abs{X}$ is the result of combinatorially gluing realizations $\abs{\square^n} = I^n$. We may therefore reduce to the case treated in 
Example~\ref{ex:productofcubes}.
\end{proof}

\begin{prop}
\label{prop:productscommute}
Topological product commutes with fibre products of cubical sets. Let $X_i \to S_i$ and $Y_i \to S_i$ for $i=1,2$ be maps of cubical sets. Then the canonical map of cubical sets
\[
(X_1 \boxtimes_{S_1} Y_1) \times (X_2 \boxtimes_{S_2} Y_2) \to (X_1 \times X_2) \boxtimes_{S_1 \times S_2} (Y_1 \times Y_2)
\]
is an isomorphism.
\end{prop}
\begin{proof}
The map exists due to the universal property of the fibre product $- \boxtimes_{S_1 \times S_2} -$ using the projections of each topological product factor on the left hand side. 

It thus suffices to compare cubes in dimension $n$ on both sides. Since parametrized $n$-cubes decompose naturally in disjoint summands for $(p,q)$ with $p+q=n$, it suffices to compare the individual summands. Moreover, the functor $- \times^{G_p \times G_q} G_n$ preserves fibre products, hence it remains to see that the canonical map 
\[
(X_{1,p} \times_{S_{1,p}} Y_{1,p}) \times (X_{2,q} \times_{S_{2,q}} Y_{2,q})
\to 
\big(X_{1,p}  \times X_{2,q}\big)  \times_{(S_{1,p}  \times S_{2,q})}  \big(Y_{1,p}  \times Y_{2,q}\big)
\]
is bijection.
\end{proof}

\begin{rmk}
Trees, and more generally graphs, are topological realizations of cube complexes of dimension $1$. Products of graphs therefore are topological realizations of topological products, hence themselves cube complexes. We will often treat trees (and graphs) as cube complexes.
\end{rmk}

\subsection{The cyclic $1$-vertex cube complex}

As an example, we are going to define a higher rank analog 
of the square complexes $\cC_{k,l}$ of \cite[6.2.3]{burger-mozes:lattices}.

\begin{defi}
We define a \textbf{finite cyclic set with orientation} as a finite set $A$ together with a partition $A = A^+ \amalg A^-$, a bijection $T : A \to A$, and an involution $-1$ of $A$ denoted by $a \mapsto a^{-1}$, 
such that
\begin{enumerate}
\item[(i)] $T$ commutes with $-1$,
\item[(ii)]  $T$ permutes each of $A^+$ and $A^-$ cyclically,
\item[(iii)] and $-1$ interchanges $A^+$ with $A^-$.
\end{enumerate}
This datum comes equipped with an orientation map $\delta : A \to \{\pm 1\}$, the fibres of which are $\delta^{-1}(*) = A^*$.
\end{defi}

\begin{rmk}
Of course, all finite cyclic sets are isomorphic to one of the following form: 
\[
A = \{a_1, \ldots, a_n, a_1^{-1}, \ldots, a_n^{-1}\}
\]
with $T(a_i^{\pm}) = a_{i+1}^{\pm}$ and indices considered modulo $n$. The involution $-1$ acts by interchanging $a_i \leftrightarrow a_i^{-1}$, and the partition $A^\pm$ sorts $A$ by exponent. Defining an abstract notion frees our notation from an index.
\end{rmk}

\begin{defi}
Let $\underline{A} = (A_1, \ldots, A_d)$ be a tuple of finite cyclic sets with orientation. The cube complex
\[
\cC_{\underline{A}}
\]
is the cubical set $\cC_{\underline{A}} : \Cube \to \sets$ with $n$-dimensional cubes
\[
\cC_{\underline{A}}([n]) =  \left\{(\tau, a) \ ; \ 
\begin{array}{c}
\tau : \{1,\ldots, n\} \inj \{1,\ldots, d\} \ \text{ injective} \\[1ex]
a : \{1, \ldots, n\} \to \coprod_i A_i  \ \text{ with }  a(i) \in A_{\tau(i)}
\end{array}
\right\} .
\]
The definition of face maps for $\cC_{\underline{A}}$ is  bit more involved: 
For each cube $(\tau,a)$ we have an orientation map $\delta: \{1,\ldots, n\} \to \{-1,1\}$ defined by
\[
\delta (i) := \delta(a(i)) = \left\{
\begin{array}{rl}
1 & \text{ if } a(i) \in A_{\tau(i)}^+, \\[1ex]
-1 & \text{ if } a(i) \in A_{\tau(i)}^-.
\end{array}
\right.
\]
The face map of $\cC_{\underline{A}}$ induced by $f: [m] \to [n]$ defined by 
\[
\sigma : \{1,\ldots, m\} \inj \{1, \ldots, n\} \quad \text{ and } \quad \ep: \{1,\ldots,n\} \to \{-1,1\}
\]
is given by $f^{\ast}(\tau,a) = (\tau \circ \sigma, b)$ with 
$b : \{1, \ldots, m\} \to \coprod_i A_i$ defined as
\begin{align*}
s(j)  & := \sum_{k \not= \tau(j) \ : \ \bar{\ep}(k) = 1} \delta(k) \\[1ex]
b(j) & := T^{s(j)} \big(a(\sigma(j))^{\ep(\tau(j))}\big) \in A_{\tau(\sigma(j))}.
\end{align*}
Here we use $\bar \ep = f(-1,\ldots, -1)$ to denote a distinct corner of the face described by $f: [n] \to [m]$. It is the corner diagonally opposite to $\ep = f(1,\ldots, 1)$. The vector $\bar \ep$ has the opposite sign compared to $\ep$ for all $i$ in the image of $\sigma$.
\[
\bar{\ep}(i) = \left\{
\begin{array}{rl}
\ep(i) & \text{ if } i \notin \im(\sigma), \\[1ex]
- \ep(i) &  \text{ if } i \in \im(\sigma).
\end{array}
\right.
\]
\end{defi}

\begin{rmk}
The cube $(\tau,a)$ of $\cC_{\underline{A}}$ has edge $a(i)$ in $i$-th direction originating in the corner $(-1, \ldots, -1)$. The remaining edges can be understood by understanding all $2$-dimensional faces: in $i \not=j$-th direction at corner $e \in \{-1,1\}^n$ (with $e_i=e_j=-1$) we have the square
\begin{equation}
\label{eq:squares of cyclic cube}
\xymatrix@M0ex@R-2ex@C-2ex{
\bullet \ar[r]^(1){T^{s(e)+\delta(j)}a(i)} & \ar@{-}[r] & \bullet  \\
\ar@{-}[u] & & \ar@{-}[u] \\
\bullet  \ar[u]^(1){T^{s(e)}a(j)}  \ar[r]_(1){T^{s(e)}a(i)} &\ar@{-}[r]  & \bullet  \ar[u]_(1){T^{s(e)+\delta(i)}a(j)} \\
}
\end{equation}
where 
\[
s(e) = \sum_{k=1}^n \frac{e(k) +1}{2} \delta(k) = \sum_{k: \ e(k) = 1} \delta(k).
\]
\end{rmk}

\begin{prop}
The cubical set $\cC_{\underline{A}}$ is a cubical complex with a single vertex. Its universal cover is a product of trees.  The local permutation group on the set $A_i$ of half edges in $i$-th direction is given by the cyclic group generated by $T$.
\end{prop}
\begin{proof}
That the action of $G_n$ on $\cC_{\underline{A}}([n])$ is free follows immediately from the definition. Also the number of vertices is obviously $1$.

The link at the unique vertex is easily seen to be the flag complex of all simplices from the set $A = \coprod_i A_i$ with at most one vertex from each $A_i$. This is precisely the link of a product of trees of valencies $m_i = \#A_i$, and also the condition for the universal covering to be a product of trees. 

The structure of the local permutation group follows by inspecting squares as in \eqref{eq:squares of cyclic cube}. An edge $a(i) \in A_i$ acts on $A_j$ for $j \not= i$ by $T^{-\delta(a(i))}$.
\end{proof}

\begin{rmk}
The fundamental group $\pi_1(\cC_{\underline{A}})$ can best be understood as a subgroup of the following group
\[
G(\underline{m}) = \left\langle a_1, \ldots, a_d, T_1, \ldots, T_d \ \left| \ 
\begin{array}{c}
T_i^{m_i} = 1, \ T_iT_j = T_j T_i  \ \text{ for all } i, j\\[1ex]
T_i a_j = a_j T_i \ \text{ for all } i \not= j \\[1ex]
(T_i a_i T_i^{-1})a_j = a_i(T_j a_j T_j^{-1}) \ \text{ for all } i \not= j 
\end{array}
\right.
\right\rangle.
\]
There is a homomorphism $\ph: G(\underline{m}) \surj \prod_{i} \bZ/m_i \bZ$ that maps all $a_i$ to $0$ and sends $T_i$ to the generator $1 \in \bZ/m_i \bZ$. The kernel of $\ph$ is canonically isomorphic to $\pi_1(\cC_{\underline{A}})$ with the $T_i$-conjugates and inverses of $a_i$ in bijection with $A_i$ such that conjugation with $T_i$ becomes the shift operator $T$ (or its inverse, depending on the convention of local permutation actions).
\end{rmk}

\section{Cube complexes and the doubling construction}
\label{sec:double}

\subsection{The type trivializing cover and the doubling construction}  
\label{sec:double_general}
The doubling construction was used by Burger and Mozes \cite[\S2]{burger-mozes:lattices} in the case of square complexes (cube complexes of dimension $2$) to construct interesting lattices on products of two trees. Here we consider the higher dimensional analogue. The combinatorial description of cube complexes as cubical sets turns out to be useful.

Let $T$ be a cube complex whose topological realization is a tree. Then there are two maps (these are called a \textbf{folding} of $T$)
\[
T \to \square^1.
\]
For any automorphism $f: T \to T$ there is a unique automorphism $\ph$ such that
\[
\xymatrix@M+1ex{
T \ar[d] \ar[r]^f & T \ar[d] \\
\square^1 \ar[r]^\ph & \square^1
}
\]
commutes. The induced homomorphism sends $f$ to its \textbf{type} $\ph = \tau(f)$:
\[
\tau: \Aut(T) \to \Aut(\square^1)  = G_1 = \{\pm 1\}.
\]

Let now $\Gamma \subseteq \Aut(T_1) \times \ldots \times \Aut(T_d)$ be a cocompact lattice in a product of trees. Then the product
\[
Z := T_{1} \times \ldots \times T_{d}
\]
is a cube complex of dimension $d$. By fixing base points $\star_i \in T_i$ we choose maps 
$T_{i} \to \square^1$ that send $\star_{i}$ to $1 \in (\square^1)_0 = \partial I  = \{1,-1\}$. The product endows $Z$ with a surjective map (again called a folding of $Z$)
\[
Z \to (\square^1)^d = \square^d,
\]
and, since $\Gamma$ acts preserving factors, we also have a type homomorphism
\[
\tau: \Gamma \to \Aut(T_{1}) \times \ldots \times \Aut(T_{d}) \to \{\pm 1\}^d = (G_1)^d \subseteq G_d.
\]
The type preserving subgroup is $\Gamma_{0} := \ker\big(\tau: \Gamma \to \{\pm 1\}^d \big)$. 
We write 
\[
Y := \Gamma_{0}\backslash Z \to \square^d
\]
for the quotient cube complex. The map $Y \to X := \Gamma \backslash Z$ is a covering space map with Galois group
\[
\Gamma/\Gamma_0 \simeq \im(\tau) \subseteq  \{\pm 1\}^d
\]
an elementary abelian $2$-group. The following Lemma is immediate.

\begin{lem}
\label{lem:surjectivetype}
If $\Gamma$ acts simply transitively on the vertices of $Z$, then $\tau : \Gamma \to \{\pm 1\}^d$ is surjective.
\end{lem}

\begin{defi}
In the above situation, and with $\Gamma$ acting simply transitively on the vertices of $Z$, we define the \textbf{doubling of $X$} as the quotient of the fibre product
\[
\cD(X) := (\Gamma/\Gamma_{0}) \backslash \big(Y \boxtimes_{\square^d} Y\big)
\]
where $\Gamma/\Gamma_{0}$ acts diagonally on both factors and via $\tau$ on $\square^d$.
\end{defi}

\begin{lem}
\label{lem:doubletreeconnected}
Let $\fg$ be a connected graph $\fg$, considered as a cube complex, with a nontrivial  map $\fg \to \square^1$. Then the graph $\fg \boxtimes_{\square^1} \fg$ is also connected.
\end{lem}
\begin{proof}
Let $\fg_0 = \fg_+ \amalg \fg_-$ be the partition of the vertices of $\fg$ given by the fibres of $\fg_0 \to (\square^1)_0 = \{+1,-1\}$. These are the equivalence classes of vertices according to the distance being even. The vertices of the fibre product are
\[
(\fg \boxtimes_{\square^1} \fg)_0 = \fg_+ \times \fg_+ \amalg \fg_- \times \fg_-.
\]
If $x,x' \in \fg_+$ have distance $2$ with common neighbour $x_1 \in \fg_-$, and $y \in \fg_+$ has neighbour $y_1 \in \fg_-$, then $(x,y)$, $(x_1,y_1)$, $(x',y)$ describes a path in 
$\fg \boxtimes_{\square^1} \fg$. It follows by induction on the distance that all points in $\fg_+ \times \fg_+$ (and likewise in $\fg_- \times \fg_-$) are connected. Clearly there is an edge linking the plus and the minus part.
\end{proof}

\begin{prop}
\label{prop:double_1vertex_general}
Let $\Gamma \subseteq \Aut(T_1) \times \ldots \times \Aut(T_d)$ be a cocompact lattice that acts simply transitively on the vertices of $Z =  \rT_{1} \times \ldots \times \rT_{d}$.  Let $X = \Gamma \backslash Z$ be the quotient cube complex. 

The doubling $\cD(X)$ is a cube complex of dimension $d$ with one vertex and with a product of trees as its universal covering space. If the tree factors of $Z$ have constant valency, then also the tree factors occurring in the universal covering of $\cD(X)$ have  constant valency.
\end{prop}
\begin{proof}
Since $\Gamma/\Gamma_{0}$ acts freely on $Y$, it also acts freely on the fibre product. Therefore
\[
Y \boxtimes_{\square^d} Y \to \cD(X)
\]
is a covering map and the local properties of $\cD(X)$ are the same as for its cover. 
The condition that only cubes of dimension $d$ are not faces of other cubes of $\cD(X)$ follows easily from the definition of the cartesian product $Y \boxtimes_{\square^d} Y$. 

Since $\Gamma$ acts vertex transitively on $Z$ by assumption, it follows that $Y \to \square^d$ is bijective on vertices. The same holds therefore for the fibre product $Y \boxtimes_{\square^d} Y \to Y$. All these vertices constitute just one orbit under $\Gamma/\Gamma_{0}$, hence $\cD(X)$ has only one vertex.

The fibre product preserves covering spaces, so it suffices to consider with 
Proposition~\ref{prop:productscommute}
\[
Z  \boxtimes_{\square^d} Z = \big(\rT_{1} \boxtimes_{\square^1} \rT_{1} \big) \times   \ldots \times \big(\rT_{d} \boxtimes_{\square^1} \rT_{d} \big).
\]
For each tree $T$ and choice of map $T \to \square^1$ the fibre product 
$T \boxtimes_{\square^1} T$ is again a graph, hence its universal covering space 
$\cT \to T \boxtimes_{\square^1} T$ is again a tree. If $T$ has constant valency, then $T$ being homogeneous, also all vertices in $T \boxtimes_{\square^1} T$ have the same valency, hence also $\cT$ is a tree of constant valency.
\end{proof}

\subsection{The doubling construction as a fibre product} 
\label{sec:doublingasfibreproduct}
The construction of the double of $X = \Gamma \backslash Z$ with notation as in 
Section~\S\ref{sec:double_general} above relies on an auxiliary cover $Y \to X$. We next describe an alternative as a genuine fibre product in the category of cubical sets. 

The group $(G_1)^d$ acts freely on the cubical set $\square^d$. We denote the quotient (in the sense of a cubical set as a presheaf on $\Cube$, hence by passing to the quotient set in each dimension) by 
\[
\Delta^d = (G_1)^d \backslash \square^d.
\]
The $p$-dimensional cubes of $\Delta^d$ are in bijection with injective maps
\[
\sigma : \{1,\ldots,p\} \inj \{1,\ldots, d\}
\]
and the face maps of $\Delta^d$ only depends on the permutation part and not on the signs. 

By the categorical properties of the quotient we obtain a commutative square of cubical sets
\begin{equation}
\label{eq:directional_marking_X}
\xymatrix@M+1ex{
Y \boxtimes_{\square^d} Y  \ar[r]^{\pr_1} \ar[d] & Y \ar[r] \ar[d] & \square^d \ar[d] \\
\cD(X) \ar[r] & X \ar[r] & \Delta^d .
}
\end{equation}

\begin{prop}
Diagram \eqref{eq:directional_marking_X} consists of two fibre product squares in the category of cubical sets. In particular
\[
\cD(X) = X \boxtimes_{\Delta^d} X.
\]
\end{prop}
\begin{proof}
By the universal property of the fibre product we obtain a map
\[
f: Y \to X \boxtimes_{\Delta^d} \square^d,
\]
lying over the identity of $X$. Since both $X$ and $\Delta^d$ are quotients by free actions of $(G_1)^d$, the fibre of $f$ above a $p$-cube $x \in X_p$ is a map of principal homogeneous $(G_1)^d$-sets. Therefore $f$ is a bijection in each dimension. The proof for  the other square uses the same argument.

The claim for the double of $X$ follows since there is a map 
\[
\cD(X) \to X \boxtimes_{\Delta^d} X
\]
by construction and the universal property of the fibre product. This map is an isomorphism, because it becomes after base change with the `unramified' cover $\square^d \to \Delta^d$ the identity
\[
Y \boxtimes_{\square^d} Y = \cD(X) \boxtimes_{\Delta^d} \square^d  \to (X \boxtimes_{\Delta^d} X) \boxtimes_{\Delta^d} \square^d  = Y \boxtimes_{\square^d} Y
\]
up to canonical identifications of fibre products.
\end{proof}

\subsection{Presentation of the fundamental group of the doubling} The \textbf{fundamental group} of a cube complex $Q$ is defined to be the fundamental group of its topological realization:
\[
\pi_1(Q) = \pi_1(\abs{Q})
\]
where we always implicitly think of a chosen base point.

Let $X$ be a cube complex as in Section~\S\ref{sec:double_general}, and let $\cD(X)$ be its double.

\begin{cor}
\label{cor:pi1ofdouble}
There is a short exact sequence
\[
1 \to \prod_{i=1}^d \pi_1(T_{i} \boxtimes_{\square^1} T_{i}) \to \pi_1(\cD(X))  \to \Gamma\times_{\Gamma/\Gamma_{0}} \Gamma \to 1.
\]
\end{cor}
\begin{proof}
This follows from Proposition~\ref{prop:double_1vertex_general}, because $Z \boxtimes_{\square^d} Z \to Y \boxtimes_{\square^d} Y \to \cD(X)$ is an unramified cover that is connected by Lemma~\ref{lem:doubletreeconnected}. Its covering group is
\[
\Aut\big(Z \boxtimes_{\square^d} Z/\cD(X) \big) = \Gamma \times_{\Gamma/\Gamma_{0}} \Gamma. \qedhere
\]
\end{proof}

The description $\cD(X) = X \boxtimes_{\Delta^d} X$ makes it easier to extract a finite presentation of $\pi_1(\cD(X))$. The map $f: X \to \Delta^d$ induces partition of edges
\[
X_1 = \coprod_{\tau = 1}^d A_{\tau}
\]
which we refer to as the partition by \textbf{direction} of the edges. Each set $A_{\tau} = f^{-1}(\tau)$ is a $G_1$-subset of $X_1$. The involution of $A_\tau$ induced by  $-1 \in G_1$ will be denoted by $a \mapsto a^{-1}$. It corresponds to the orientation reversion of the edge. The group $\pi_1(X)$  is generated by the union $X_1 = A_1 \cup \ldots \cup A_d$ and $a^{-1}$ is indeed the inverse of $a \in X_1$.

The set of squares of $X$ is similarly partitioned as
\[
X_2 = \coprod_{\tau \not= \sigma}  R_{\tau,\sigma} 
\]
with $R_{\tau,\sigma} = f^{-1}( i_{\tau,\sigma})$,
where $i_{\tau,\sigma}: \{1,2\} \inj \{1,\ldots, d\}$ denotes the square of $\Delta^d$ 
which sends $1 \mapsto \tau$ and $2 \mapsto \sigma$. The face maps of $\Delta^d$ have the effect
\[
\delta^{\pm}_1( i_{\tau,\sigma} ) = \tau, \qquad \delta^{\pm}_2( i_{\tau,\sigma} ) = \sigma.
\]

Each square $\rho \in R_{\tau, \sigma}$  induces a relation in $\pi_1(X)$ 
\[
\delta^+_1(\rho) \delta^-_2(\rho) = \delta^+_2(\rho) \delta^-_1(\rho) 
\]
where $\delta^{\pm}_1(\rho) \in A_\tau$ and $\delta^{\pm}_2(\rho) \in A_\sigma$.
The relations in a $G_2$-orbit of squares are equivalent. If we set
\[
\delta^{+}_1(\rho) = a, \ \delta^{-}_1(\rho) = a'^{-1} \quad \text{ and } \quad 
\delta^{+}_2(\rho) = b'^{-1}, \ \delta^{-}_2(\rho) = b,
\]
then the relation associated to $\rho$ is
\[
ab a'b',
\]
and the relations for the squares in the full $G_2$-orbit of $\rho$ are (we set $\bar \rho$ for the image of $\rho$ under the map $[2] \to [2]$ that flips the two entries with all signs $+1$):
\[
{\renewcommand{\arraystretch}{1.2}
\setlength{\arraycolsep}{1em} 
\begin{array}{c|c}
\toprule
\rho \leadsto ab a' b'  & (-1,1)^\ast\rho \leadsto a'^{-1}b^{-1}a^{-1}b'^{-1}  \\[1ex]  \midrule
(1,-1)^\ast \rho \leadsto a^{-1} b'^{-1} a'^{-1} b^{-1} & (-1,-1)^\ast \rho \leadsto a'b'ab \\[1ex] \midrule
\bar{\rho} \leadsto b'^{-1} a'^{-1} b ^{-1} a^{-1} & (-1,1)^\ast\bar{\rho} \leadsto ba'b'a \\[1ex]\midrule
(1,-1)^\ast\bar{\rho} \leadsto  b'aba' & (-1,-1)^\ast\bar{\rho} \leadsto  b^{-1}a^{-1}b'^{-1}a'^{-1}.
\\ \bottomrule
\end{array}
}
\]
In particular, the relations for $R_{\tau,\sigma}$ are equivalent to the relations for $R_{\sigma,\tau}$. 

\begin{prop}
\label{prop:presentation_1-vertexcubecomplex}
With the above notation, we have
\[
\pi_1(X) = \left\langle A_1 \cup \ldots \cup A_d \ \left| \
\begin{array}{cc}
aa^{-1}  = 1 & \text{ for all } a \in A_1 \cup \ldots \cup A_d \\
\delta^+_1(\rho) \delta^-_2(\rho) = \delta^+_2(\rho) \delta^-_1(\rho)  & \text{ for all } \rho \in \bigcup R_{\tau, \sigma}
\end{array}
\right.\right\rangle.
\]
\end{prop}
\begin{proof}
For the fundamental group of a $1$-vertex cube complex, the generators correspond to the set of edges with the orientation inversion being equal to the inversion in $\pi_1$. The relations are all generated by relations coming from squares, hence agree with the relations noted in the above presentation.
\end{proof}

Since $\cD(X)$ is also a cube complex with just a single vertex and a structure map 
\[
\cD(X) \to \Delta^d,
\]
we simply have to determine the partition of edges and relate edges and squares  of $\cD(X)$ to edges and squares of $X$. We immediately get the following finite presentation:

\begin{prop}
\label{prop:presentation_double}
With the above notation, we have
\[
\pi_1(\cD(X)) = \left\langle \bigcup_{i=1}^d A_i \times A_i  \left| 
\begin{array}{c}
(a,a') (a^{-1},a'^-{1})  = 1 \qquad \text{ for all } (a,a') \in A_\tau \times A_\tau  \text{ and all } \tau \\[1ex]
\big(\delta^+_1(\rho),\delta^+_1(\rho')\big)\big(\delta^-_2(\rho),\delta^-_2(\rho')\big) = \big(\delta^+_2(\rho),\delta^+_2(\rho')\big) \big(\delta^-_1(\rho),\delta^-_1(\rho')\big)  \\[1ex]
\text{ for all } (\rho,\rho') \in  R_{\tau, \sigma}  \times R_{\tau,\sigma} \text{ and all } \tau \not= \sigma
\end{array}
\right.\right\rangle.
\]
\end{prop}
\begin{proof}
We apply Proposition~\ref{prop:presentation_1-vertexcubecomplex} to $\cD(X)$ using that
\[
\cD(X)_1 = (A_1 \times A_1) \cup \ldots \cup (A_d \times A_d)
\]
and $\cD(X)_2$ is the union of the $R_{\tau, \sigma}  \times R_{\tau,\sigma}$ with face maps defined componentwise. 
\end{proof}

\subsection{Non-residual finite lattices on products of trees in higher rank}

We resume the discussion of the quaternionic lattice $\Gamma_S$ from Section~\S\ref{sec:lattice}, where  
\[
\{0,\infty\} \subseteq S = \{0,\infty,v_1,\ldots, v_d\} \subseteq \bP^1(\bF_q)
\]
is a finite set of rational places. The corresponding product of trees
\[
\PT_{S_0} = \BTT_{v_1} \times \ldots \times \BTT_{v_d}
\]
is a cube complex of dimension $d$. Our choice of a sheaf of maximal orders $\cA$ determines a standard vertex $\star_{v_i} \in \BTT_{v_i}$ and therefore fixes a folding map
\[
\BTT_{v_i} \to \square^1, \quad \star_{v_i} \mapsto 1 \in (\square^1)_0.
\]
It follows that $\PT_{S_0}$ comes equipped with a folding map $\PT_{S_0} \to (\square^1)^d = \square^d$. Since $\Gamma_S$ acts vertex transitively on $\PT_{S_0}$, we are in the situation of 
Section~\S\ref{sec:double_general} and the double 
\[
\cD(X_S) = (\Gamma_S/\Gamma_{S,0}) \backslash \big(Y_S \boxtimes_{\square^d} Y_S\big)
\]
of $X_S = \Gamma_S \backslash \PT_{S_0}$ exists 
(notation parallel to Section~\S\ref{sec:double_general}).

\begin{prop}
The doubling $\cD(X_S)$ is a cubical complex of dimension $d$ with one vertex and with a product of trees of constant valency as its universal covering space. There is a short exact sequence
\[
1 \to \prod_{i=1}^d \pi_1(\BTT_{v_i} \boxtimes_{\square^1} \BTT_{v_i}) \to \pi_1(\cD(X_S))  \to \Gamma_S \times_{\Gamma_S/\Gamma_{S,0}} \Gamma_S \to 1.
\]
\end{prop}
\begin{proof}
This is Proposition~\ref{prop:double_1vertex_general} and Corollary~\ref{cor:pi1ofdouble}.
\end{proof}

\begin{rmk}
Wise in his thesis \cite[Thm 5.5]{wise:thesis} constructs a non-residually finite lattice on a product of two trees. In a situation similar to ours, Burger and Mozes \cite[Cor 2.5]{burger-mozes:lattices} find non-residual finite fundamental groups of $1$-vertex square complexes under a condition on the local permutation actions. Burger and Mozes make use of the doubling construction in dimension $2$.  
\end{rmk}

Our result is a higher rank analog of the non-residually finiteness of the doubling construction for square complexes. Indeed, we use the higher rank version of the doubling construction established in  Section~\S\ref{sec:double}, and then deduce non-residual finiteness from a $2$-dimensional subcomplex which is the double of a $1$-vertex square complex. We need the following well known lemma and apply it to non positively curved cube complexes which are locally CAT(0).

\begin{lem}
\label{lem:localisometry_injectivepi1}
Let $f: X \to Y$ be a locally isometric map between connected, locally CAT(0) spaces. We fix a base point $x \in X$ and set $y = f(x)$. Then the induced map $f_\ast: \pi_1(X,x) \to \pi_1(Y,y)$ is injective.
\end{lem}
\begin{proof}
Let $g \in \pi_1(X,x)$, $g\neq 1$ and $\gamma$ a loop in $X$ representing $g$.  We must show that $f_\ast(g) \not= 1$. Note that  $\gamma$ is freely homotopic to a locally geodesic loop $\beta$. Since $f$ is a local isometry, it sends local geodesics to local geodesics and $f(\beta)$ is a locally geodesic loop. 

Nontrivial locally geodesic loops are not nullhomotopic in locally CAT(0) spaces.  Otherwise the locally geodesic loop would lift to the universal covering space, a genuine CAT(0) space in which locally geodesics are true geodesics, a contradiction to the periodicity of the loop. 

But $f_\ast(g) = 1$  holds if and only if $f(\beta)$ is (freely) nullhomotopic in $Y$. This proves the lemma.
\end{proof}

\begin{prop}
\label{prop:arithmeticdouble}
Let $S \subseteq \bP^1$ be a finite set of $\bF_q$-rational places containing $B = \{0,\infty\}$  and set $S_0 = S \setminus B$. 

Let $\Gamma_S$ be the lattice of Theorem~\ref{thm:structure_Gamma_S} acting on the product of trees $\PT_{S_0}$. The double $\cD(X_S)$ of 
\[
X_S = \Gamma_S \backslash \PT_{S_0}
\]
is a one vertex cube complex with product of trees as universal cover and fundamental group $\pi_1(\cD(X_S))$ of the following presentation:
\begin{enumerate}
\item generators are pairs  $([1+\alpha F],[1+\alpha' F])$ with $N(\alpha)^{-1} = N(\alpha')^{-1} \in S_0$,
\item
relations for inverses:
\[
([1+\alpha F],[1+\alpha' F]) \cdot ([1-\alpha F],[1-\alpha' F]) = 1,
\]
\item and relations from geometric squares: for all $\tau \not= \sigma \in S_0$ and $\alpha,\alpha'$ with $N(\alpha)^{-1} = N(\alpha')^{-1} = \tau$, and $\beta,\beta'$ with  $N(\beta)^{-1} = N(\beta')^{-1} = \sigma$ we have
\[
([1 + \alpha F],[1+\alpha' F]) \cdot ([1+ \beta F],[1+\beta' F]) = ([1 + \zeta_{\alpha}(\beta) \beta F],[1 + \zeta_{\alpha'}(\beta') \beta' F]) \cdot ([1  + \zeta_\beta(\alpha) \alpha F],[1  + \zeta_{\beta'}(\alpha') \alpha' F]).
\]
\end{enumerate}
\end{prop}
\begin{proof}
Proposition~\ref{prop:double_1vertex_general} shows that the lattice $\pi_1(\cD(X_S))$ in question acts on a product of trees with a one vertex cube complex as a quotient. The  claimed presentation for $\pi_1(\cD(X_S))$ follows from Theorem~\ref{thm:structure_Gamma_S} 
together with Proposition~\ref{prop:presentation_double}.
\end{proof}

\begin{thm}
\label{thm:non residually finite doubles}
In the context of Proposition~\ref{prop:arithmeticdouble} we assume $\#S_0 \geq 2$. Then the lattice $\pi_1(\cD(X_S))$ is not residually finite.
\end{thm}
\begin{proof}
Let $\tau \not = \sigma \in S_0$ and consider the analog construction for the set $\{0,\infty, \tau, \sigma\}$. Then $X_{\tau,\sigma} := X_{\{0,\infty, \tau, \sigma\}}$ is a subcomplex of $X_S$ via adding the standard vertex in the remaining tree factors. Since the doubling is functorial we also obtain an embedding  $\cD(X_{\tau,\sigma}) \subseteq \cD(X_S)$
which by Lemma~\ref{lem:localisometry_injectivepi1} gives rise to an injective map
\[
\pi_1\big(\cD(X_{\tau,\sigma})\big) \inj \pi_1\big(\cD(X_S)\big).
\]
This reduces the claim of the theorem to the case $\#S_0 = 2$, which is dealt with in Proposition 4.15 of \cite{caprace:survey} based on a result of 
Caprace and Monod \cite[Prop. 2.4]{caprace_monod:non-residualfiniteness}. 
\end{proof}

\section{Ramanujan complexes}
\label{sec:ramanujan}
\subsection{Adjacency operators for graphs and Ramanujan graphs} 
Let $X$ be a connected graph with uniformly bounded valencies. We consider $X$ as a $1$-dimensional cubical complex and write $X_0$ for the set of vertices of $X$. We write $P \sim Q$ if two vertices $P,Q \in X_0$ are adjacent, and we denote by $\mu(P,Q)$ the number of edges that connect $P$ with $Q$. The \textbf{adjacency operator} $A_X$ acting on the space of $L^2$-functions $f : X_0 \to \bC$ is defined as 
\[
A_X(f)(P)= \sum_{Q  \sim P} \mu(P,Q) f(Q),
\]
where we sum over adjacent vertices with the multiplicity of the number of edges linking them.
The adjacency operator commutes with the induced right action of the group $\Aut(X)$ of graph automorphisms on $L^2(X_0)$.

\begin{rmk}
For the moment, let $X$ be a finite graph of constant valency $q+1$. The \textbf{trivial eigenvalues} of $A_X$ acting on $L^2(X_0)$ are $\lambda = \pm (q+1)$. These are obtained by the constant non-zero function for $\lambda = q+1$, and by the `alternating function' with $f(P) = - f(Q) \not= 0$ for all $P \sim Q$ for $\lambda = -(q+1)$. The latter only exists if $X$ has a bipartite structure, i.e., as a  cube complex $X$ admits a folding $X \to \square^1$. 
\end{rmk}

\begin{rmk}
Alon and Boppana \cite{alon-boppana} prove that asymptotically in families of finite $(q+1)$-regular graphs $X_n$ with diameter tending to $\infty$ the largest absolute value of a non-trivial eigenvalue $\abs{\lambda(X_n)}$ of the adjacency operator $A_{X_n}$ has limes inferior
\[
\varliminf_{n \to \infty} \abs{\lambda(X_n)} \geq 2\sqrt{q}.
\]
Lubotzky, Phillips and Sarnak turn this estimate into a definition for the extremal case as follows. 
\end{rmk}

\begin{defi}[{\cite[Definition 1.1]{LPS}}]
A finite $(q+1)$-regular graph $X$ is defined to be a \textbf{Ramanujan graph}  
if all non-trivial eigenvalues $\lambda$ of the adjacency operator $A_X$ have absolute value 
$\abs{\lambda} \leq 2\sqrt{q}$. 
\end{defi}

\subsection{Adjacency operators for augmented cubical complexes} 
\label{sec:v adjacency}
Recall the cube complex $\Delta^d = (G_1)^d \backslash \square^d$ from 
Section~\S\ref{sec:doublingasfibreproduct}, whose $p$-dimensional cubes are in bijection with injective maps
\[
\sigma : \{1,\ldots,p\} \inj \{1,\ldots, d\}.
\]
We will consider $d$-dimensional cubical complexes $X$ together with an augmentation 
\[
X \to \Delta^d
\]
that we may think of a consistent way to associate to a $p$-dimensional cube of $X$ an ordered set of `directions' taken from a subset of $\{1,\ldots, d\}$ of size $p$. For adjacent vertices $P,Q \in X_0$ the edge joining $P$ with $Q$ has a well defined direction $v \in \{1,\ldots, d\}$. We denote by $P \sim_v Q$ if $P$ and $Q$ are \textbf{adjacent in $v$-direction}, i.e., the edge linking $P$ with $Q$ has direction $v$. We denote by $\mu_v(P,Q)$ the number of edges in $v$-direction that connect $P$ with $Q$.

We now assume that $X$ satisfies a condition of locally finiteness: all vertices are faces of at most a uniformly bounded number of  edges. Then, for every $v \in \{1,\ldots, d\}$, we can define an \textbf{adjacency operator $A_v$ in $v$-direction} on $L^2(X_0)$ by 
\[
A_v(f)(P) = \sum_{Q  \sim_v P} \mu_v(P,Q) f(Q). 
\]
The operators $A_v$ commute with the automorphisms of $L^2(X_0)$ induced by  automorphisms of $X$ preserving the augmentation $X \to \Delta^d$. 

\begin{rmk}
\label{rmk:comuting v adjacency}
If we further assume that locally all pairs of edges starting in a vertex $P$ in directions $v \not= w$ belong to a unique square in $X_2$ (in direction $\{v,w\}$), then the operators $A_v$ and $A_w$ commute with each other. Indeed, the products $A_vA_w$ and $A_wA_v$ both sum over all values of secondary neighbours that are one step away both in $v$ and $w$-direction.  
\end{rmk}

\begin{rmk}
For the moment, let $X \to \Delta^d$ be an augmented finite cube complex and for all $v \in \{1, \ldots, d\}$ there is a $q_v$ such that every vertex $P \in X_0$ has $q_v + 1$ neighbours in $v$-direction. We say that $X$ has \textbf{constant valency $(q_v+1)_v$ in all directions}.

The \textbf{trivial eigenvalues} of $A_v$ acting on $L^2(X_0)$ are $\lambda = \pm (q_v+1)$. These are obtained by non-zero functions that are constant in $v$-direction for $\lambda = q_v+1$, and by the `$v$-alternating function' with $f(P) = - f(Q) \not= 0$ for all $P \sim_v Q$ for $\lambda = -(q_v+1)$. The latter only exists if $X$ has a kind of bipartite structure in $v$-direction. Existence for all $v$ means that the augmentation lifts to a folding $X \to \square^d$.
\end{rmk}

\begin{defi}
Let $X \to \Delta^d$ be a finite cubical complex of dimension $d$ that  has constant valency $(q_v+1)_v$ in all directions. Then $X$ is a \textbf{cubical Ramanujan complex}, if for each $v \in \{1, \ldots, d\}$, the eigenvalues $\lambda$ of $A_v$ are trivial, i.e., $\lambda = \pm(q_v+1)$, or non-trivial and then bounded by 
\[
\abs{\lambda} \leq 2\sqrt{q_v}.
\]
\end{defi}

\begin{rmk}
Our definition of cubical Ramanujan complexes follows Jordan and Livn\'e \cite[Definition 2.5]{JL}, which is motivated by their Alon-Boppana type result \cite[Proposition 2.4]{JL}. As in \cite{JL}, a generalization of our definition applies to spectrum of an action on cellular cochains of $X$ in degree $> 0$.

For simplicial versions of Ramanujan complexes see also \cite{csz} and \cite{LSV,LSV1}, and for recent surveys on higher expanders we refer to \cite{lubotzky:takagilecture,lubotzky:highexpanders}.
\end{rmk}

\subsection{Hecke action and adjacency operator} 
\label{sec:adjacency is Hecke}
We now work in the context and notation of Section~\S\ref{sec:lattice}, with the exception\footnote{There is no more explicit Frobenius map or computation in a quaternion algebra, and we need the letter $K$ for compact open subgroups.} that we now denote the global field by $F$. So $F$ is a global field, $D/F$ is a quaternion algebra ramified in the set of places $B$ contained in a finite set of places $S$. We set $S_0 = S \setminus B$. At this moment it is not necessary to restrict to a global function field $F$ over $\bF_q$, or to ask the absolute norms $q_v := N(v)$ of the places $v \in S_0$ to be all equal to $q$.
For $v \in S_0$, as in Notation~\ref{nota:lamdaSaslatticeinlocallycompactgroup}, we consider the Bruhat-Tits building $\BTT_v$, a regular tree of valency $q_v + 1$, of the locally pro-finite group
\[
{G}_v := \PGL_2(F_v) \simeq \PGL_{1,\dA}(F_v).
\]
The choice of the maximal order $\dA$ distinguishes a vertex $\star_v \in \BTT_v$. The stabilizer of $\star_v$ agrees with a maximal compact subgroup $K_v \subseteq {G}_v$. Let $t_v \in F_v$ be a uniformizer. If $F$ is a global function field as in Section~\S\ref{sec:lattice}, then $F_v \simeq \bF_{q_v}((t_v))$, and with a suitable choice of isomorphism  ${G}_v  \simeq \PGL_2\big(\bF_{q_v}((t_v))\big)$, we have
\[
K_v := \PGL_2(\mathbb{F}_{q_v}[[t_v]]) \subseteq {G}_v = \PGL_2\big(\mathbb{F}_{q_v}((t_v))\big).
\]
The set $\BTT_{v,0}$ of vertices of $\BTT_v$ is in ${G}_v$-equivariant bijection with ${G}_v/K_v$ via the map ${G}_v/K_v \to \BTT_{v,0}$, $g \mapsto g\star_v$. So, the adjacency operator $A_{\BTT_v}$ acts on 
\[
L^2(\BTT_{v,0}) = L^2({G}_v/K_v) = \{ f \in L^2({G}_v) \ ; \ f \text{ right $K_v$-invariant} \}
\]
and as such has the following well known explicit description as a Hecke operator from the spherical Hecke algebra $\dH(G_v \!\sslash\! K_v)$ of compactly supported $K_v$-biinvariant functions on $G_v$ with multiplication by convolution. 
We consider the diagonal matrix 
\[
a_v = \matzz{t_v}{}{}{1} \in G_v.
\]
Under the bijection $G_v/K_v = \BTT_{v,0}$ the double coset $K_v a_v K_v$ maps to the set of neighbours of the distinguished vertex $\star_v$:
\[
K_v a_v K_v/K_v = \{P \in V(\BTT_v) \ ; \ P \sim \star_v\} =: \rS^1(\star_v).
\]
Therefore the corresponding Hecke operator defined by the double coset $K_v a_v K_v$ (and with respect to a suitably normalized right-invariant Haar measure of $G_v$; here $ \one_{K_v a_v K_v}(x)$ is the characteristic function of $K_v a_v K_v$)
\[
f|_{[K_v a_v K_v]} (g) = \int_{G_v}  f(gx^{-1})  \one_{K_v a_v K_v}(x) d\mu_{G_v}(x) = \sum_{P \sim \star_v} f(gP) = A_{\BTT_v}(f)(g)
\]
is nothing but the adjacency operator $A_{\BTT_v}$, because $g \rS^1(\star_v) =  \rS^1(g\star_v)$ is the set of neighbours of $g\star_v$. Note that for a tree all non-zero multiplicities are equal to $1$.

The spherical Hecke algebra is isomorphic via the Satake isomorphism, see \cite[\S8 Thm 7]{satake}, to the polynomial ring 
\[
\dH(G_v \!\sslash\! K_v)  \simeq \bC[A_{\BTT_v}].
\]
The adjacency operator $A_{\BTT_v}$ commutes with the automorphisms of $L^2(G_v/K_v)$ induced by left translation with $g \in G_v$.

\subsection{Hecke action and adjacency operators for products of trees} 
We set $d = \#S_0$. As in Section~\S\ref{sec:lattice}, we now consider the product building as an augmented $d$-dimensional cubical complex 
\[
\PT_{S_0} := \prod_{v \in S_0} \BTT_v \to \prod_{v \in S_0} \square^1  \simeq \square^d \to \Delta^d.
\]
The group $G= \prod_{v \in S_0} G_v$ acts transitively on the vertices $\PT_{S_0,0}$  with stabilizer $K = \prod_{v \in S_0} K_v$ of the distinguished vertex $\star = (\star_v)_{v \in S_0} \in \PT_{S_0,0}$. We can therefore identify $\PT_{S_0,0}$ with $G/K$.

We write $P \sim_v Q$ if two vertices $P,Q \in \PT_{S_0,0}$ are adjacent in $v$-direction, i.e., only the entries with index $v$ differ and are adjacent in $\BTT_v$.  Consistently with the definition in Section~\S\ref{sec:v adjacency} we define an \textbf{adjacency operator $A_v$ in $v$-direction} on $L^2(G/K) = L^2(\PT_{S_0,0})$ for every $v \in S_0$ by 
\[
A_v(f)(P) = \sum_{Q  \sim_v P} f(Q). 
\]
Note again, that all non-zero multiplicities $\mu_v(P,Q)$ are equal to $1$.

\begin{rmk}
Clearly, the operators $A_v$ and $A_w$ for $v,w \in S_0$ commute with each other, because $A_vA_w$ and $A_wA_v$ both sum over all values of secondary neighbours that are one step away both in $v$ and $w$-direction, see Remark~\ref{rmk:comuting v adjacency}.  In fact, this also follows from the interpretation as Hecke operators, since the Hecke algebra turns out to be commutative.
Similarly to the case of the adjacency operator of just one tree $\BTT_v$ treated in the previous section, the $v$-directional adjacency operators $A_v$ have interpretations as Hecke operators for $G/K$ by the $K$-double coset of 
\[
(1, \ldots, 1, a_v, 1 \ldots, 1) \in G = \prod_{v \in S_0} G_v.
\]
In fact, the (spherical) Hecke algebra  can be identified with  the polynomial algebra  with generators $A_v$, for $v \in S_0$:
\[
\dH(G\!\sslash\! K) = \bigotimes_{v \in S_0} \dH(G_v \!\sslash\! K_v) \simeq \bC[A_v \ ; \ v \in S_0].
\]
\end{rmk}

\begin{ex}
Let $\Gamma$ be a torsionfree cocompact lattice in $G$, then the quotient $X_\Gamma := \Gamma \backslash \PT_{S_0}$ is a finite cubical complex. Since left translation by the group respects the augmentation $\PT_{S_0}  \to \Delta^d$, we obtain a map 
\[
X_\Gamma \to \Delta^d.
\]
With $q_v = N(v)$ the finite cubical complex $X_\Gamma$ has \textbf{constant valency $(q_v+1)_v$ in all directions}. Hence there are $v$-directional adjacency operators on 
\[
L^2(X_{\Gamma,0}) 
= \{ f \in L^2(\Gamma \backslash G) \ ; \ f \text{ is right $K$-invariant}\} \subseteq L^2(\Gamma \backslash G).
\]
Moreover, a computation similar to that of Section~\S\ref{sec:adjacency is Hecke} shows that the operator $A_v$ on $L^2(X_{\Gamma,0})$ is the Hecke operator associated to  the $K$-double coset of $(1, \ldots, 1, a_v, 1 \ldots, 1) \in G$ for the $G$-representation $L^2(\Gamma \backslash G)$ by right translation; and $L^2(X_{\Gamma,0})$ is the corresponding Hecke module of $K$-invariant vectors (also called the $K$-spherical vectors).
\end{ex}

\subsection{Hecke spectrum for congruence quotients}

Our lattices $\Gamma_S$ in Section~\S\ref{sec:lattice} are $S$-arithmetic lattices contained in $\Lambda_S = \PGL_{1,\dA}(\fo_{F,S})$, see Section~\S\ref{sec:globaltorsioninlattice} for notation and Section~\S\ref{sec:adjacency is Hecke} for a modification. 

\begin{rmk}
Congruence subgroups of $\Gamma_S$ are defined as follows. For an ideal $I \lhd \fo_{F,S}$, the residue ring $\fo_{F,S}/I$ is finite, and 
\[
\PGL_{1,\dA}(\fo_{F,S}/I) \simeq \PGL_2(\fo_{F,S}/I).
\]
Hence the kernel $\Lambda_S(I)$ of \textit{evaluation mod $I$}
\[
\Lambda_S(I) := \ker\Big(\ev_I: \PGL_{1,\dA}(\fo_{F,S}) \to  \PGL_{1,\dA}(\fo_{F,S}/I) \simeq \PGL_2(\fo_{F,S}/I)\Big)
\]
is of finite index in $\Lambda_S$ and called the \textbf{principal congruence subgroup of level $I$}. Since $S_0$ is non-empty, strong approximation for $\SL_{1,\dA}$ shows that 
\[
\PSL_2(\fo_{F,S}/I) \subseteq \Lambda_S/\Lambda_S(I) \simeq \im(\ev_I)
\subseteq \PGL_2(\fo_{F,S}/I).
\]
The precise image of $\ev_I$
depends on $\dA(\fo_{F,S})$ containing elements of reduced norm that are (non-)squares modulo $I$. We shall not pursue this question here. 
The principal congruence subgroups of $\Gamma_S$ are 
\[
\Gamma_S(I) = \Gamma_S \cap \Lambda_S(I).
\]
Since $\Gamma_S$ is described in $\Lambda_S$ by congruence conditions at $t=0$, strong approximation shows that also 
\[
\PSL_2(\fo_{F,S}/I) \subseteq \Gamma_S/\Gamma_S(I).
\]
For prime level, i.e., $I = \fp$ is a prime ideal of $\fo_{F,S}$ corresponding to a place of $\bP^1_{\bF_q}$ not in $S$, we find 
\[
\PSL_2(N(\fp)) \subseteq \Gamma_S/\Gamma_S(\fp) \subseteq \PGL_2(N(\fp)).
\]

As usual, a \textbf{congruence subgroup} of $\Gamma_S$ is a subgroup $\Gamma \subseteq \Gamma_S$ that contains a principal congruence subgroup.
\end{rmk}

\begin{ex}
The quotient $X_S = \Gamma_S \backslash \PT_{S_0}$ is a cubical Ramanujan complex for trivial reasons: the complex has one vertex and thus all functions on vertices are constant. All Hecke operators $A_v$ have only the trivial eigenvalue $q_v + 1$. (Here it is important that we have defined the adjacency operators with multiplicity.)
\end{ex}

We now replace $\Gamma_S$ by a congruence subgroup. Closely following \cite{LSV1} we assert the following ample supply of cubical Ramanujan complexes of arithmetic origin.

\begin{thm}
\label{thm:Ramanujan}
Let $\Gamma_S$ be one of the lattices of finite characteristic $\not=2$ constructed in 
Section~\S\ref{sec:lattice}, and let $\Gamma \subseteq \Gamma_S$ be a congruence subgroup. Then the quotient $X_\Gamma = \Gamma \backslash \PT_{S_0}$ is a cubical Ramanujan complex.
\end{thm}

\begin{conj}
\label{conj:Ramanujan}
Moreover, infinitely many of the Ramanujan complexes of Theorem~\ref{thm:Ramanujan} are higher-dimensional coboundary expanders of bounded degree in the sense of \cite{gromovandco}.
\end{conj}

\begin{rmk}
Let $I \lhd \fo_{F,S}$ be a level coprime to $S$. The corresponding cubical complex 
\[
X_S(I) := \Gamma_S(I) \backslash \PT_{S_0} 
\]
has an explicit group theoretic description as follows. Since $\Gamma_S$ acts simply transitively on the vertices of $\PT_{S_0}$ the vertices of $X_S(I)$ are naturally
\[
{X_S(I)}_0 = \Gamma_S(I) \backslash \Gamma_S = \Gamma_S/\Gamma_S(I).
\]
Recall the natural generating set $\bigcup_{\tau \in S_0} PA_\tau$ of $\Gamma$ from 
Theorem~\ref{thm:structure_Gamma_S}. The $1$-skeleton of $X_S(I)$ becomes under the above identification the Cayley graph of the quotient $\Gamma_S/\Gamma_S(I)$ with respect to the image of $\bigcup_{\tau \in S_0} PA_\tau$. Remembering the individual $PA_\tau$ leads to a local structure of edges of $X_S(I)$ in $\tau$-direction. The full cubical complex is then obtained by gluing in all cubes that are compatible with the directional structure of the edges of the $1$-skeleton. This is possible in a unique way as soon as the diameter of $X_S(I)$ in each direction is large enough ($\geq 2$ suffices).

The above fairly explicit description of a cubical complex, cf.\ Section~\S\ref{sec:concretemodel}, is only available because $\Gamma_S$ is torsion free and acts vertex transitively on $\PT_{S_0}$. Moreover, the resulting cubical Ramanujan complex $X_S(I)$ has a vertex transitive group of automorphisms since the $\Gamma_S$-action on $\PT_{S_0}$ induces an action by $\Gamma_S/\Gamma_S(I)$ on $X_S(I)$.
\end{rmk}

\subsection{Proof of Theorem~\ref{thm:Ramanujan}} The overall proof strategy
follows closely the strategy used in \cite{margulis, LPS}, 
see also \cite{lubotzky:expanderbook, lubotzky:ramanujan}. 
The same strategy was employed by Jordan and Livn\'e in \cite{JL}.

\subsubsection{Hecke eigenvector}
Let $X_\Gamma = \Gamma \backslash \PT_{S_0}$ be the finite cubical complex associated to the congruence subgroup $\Gamma \subseteq \Gamma_S$ as in 
Theorem~\ref{thm:Ramanujan}. 

We must control the spectrum of the adjacency operators $A_v$ acting on $L^2(X_{\Gamma,0})$. Since the Hecke algebra $\dH(G\!\sslash\! K)$ containing the $A_v$ is commutative,  their spectrum on $L^2(X_{\Gamma,0})$ is described by the eigenvalues for simultaneous eigenvectors $f_\omega$ for all Hecke operators, i.e., a $1$-dimensional Hecke submodule $\bC f_\omega \subseteq L^2(X_{\Gamma,0})$. 
We consider this eigenvector as a left $\Gamma$ and right $K$-invariant function 
\[
f_\omega : G \to \bC,
\]
where being a simultaneous Hecke eigenvector means that there is a Hecke character
\[
\omega : \dH(G\!\sslash\! K) \to \bC
\]
such that for all $v \in S_0$ and $\alpha \in \dH(G\!\sslash\! K)$ we have 
\[
\alpha (f_\omega)  = \omega(\alpha) \cdot f_\omega.
\]
We consider the smooth admissible $G$-representation  
\[
V_\omega := \langle g(f_\omega) ; g \in G \rangle_\bC \subseteq L^2(\Gamma \backslash G)
\]
spanned by translates of $f_w$. Then $V_\omega$ is irreducible because its Hecke module $\bC f_\omega = V_\omega^K$ is irreducible. 

\subsubsection{Product decomposition}
By a theorem of Flath \cite{flath}, the product decomposition $G = \prod_{v \in S_0} G_v$ implies a tensor product decomposition
\[
V_\omega = \otimes_{v \in S_0} \pi_v
\]
with $\pi_v$ irreducible admissible $K_v$-spherical representations for $G_v$, for all $v \in S_0$. The eigenvalues of $A_v$, our only concern actually, by virtue of being the Hecke operator associated to 
\[
K_v \matzz{t_v}{}{}{1}K_v
\]
are nothing but the Satake parameters of the local $G_v$-representations $\pi_v$.

\subsubsection{Non-trivial eigenvalues means infinite dimensional}
If $\dim(\pi_v) < \infty$ is finite dimensional, then $G_v \simeq \PGL_2(F_v)$ acts through a commutative quotient. Hence the action map factors over 
\[
\ov{\det} : G_v \simeq \PGL_2(F_v) \to F_v^\times/(F_v^\times)^2.
\]
Being irreducible and essentially a representation of a finite abelian group, the space $\pi_v$ must be $1$-dimensional. As it also contains a $K_v$-spherical (invariant) vector, we conclude that the $G_v$ action is either trivial or by $-1$ to the power the $v$-valuation of the determinant modulo squares: the shift of type in the tree $\BTT_v$. The first case means that $A_v$ has trivial eigenvalue $q_v  +1$ and the second case means that 	we have trivial eigenvalue $-(q_v+1)$. With the same proof as in \cite[Prop 3.3 and 3.6]{LSV1}, essentially because of strong approximation for $\SL_{1,\dA}$,  we obtain the following proposition.

\begin{prop}
\label{prop:trivialeigenvalues}
With the notation above the following are equivalent:
\begin{enumerate}
\item[(a)] $A_v$ has trivial eigenvalue acting on $f_\omega$ for one $v \in S_0$,
\item[(a')] $A_v$ has trivial eigenvalue acting on $f_\omega$ for all $v \in S_0$,
\item[(b)] $\dim(\pi_v)$ is finite for one $v \in S_0$,
\item[(b')] $\dim(\pi_v)$ is finite for all $v \in S_0$,
\item[(c)] $\dim(V_\omega)$ is finite.
\end{enumerate}
\end{prop}
\begin{proof}
We only have to show (a) implies (c). If $A_v$ has trivial eigenvalue acting on $f_\omega$, then $f_\omega$ is either constant or $v$-alternating. In particular, $G_v$ acts through an abelian quotient on all of $V_\omega$ while $\Gamma$ acts trivially. Therefore the set of products $\Gamma G_v \subset G$ acts with commuting operators, hence by continuity its closure acts via an abelian group of operators. The topological and group closure in $G$ has finite index in $G$, so that $G$ acts through a finite group on $V_\omega$. This proves (c).
\end{proof}

In view of Proposition~\ref{prop:trivialeigenvalues} we must show that for infinite dimensional irreducible admissible $V_\omega = \otimes \pi_v \subseteq L^2(\Gamma \backslash G)$ the Satake parameters of all local components $\pi_v$ are bounded in absolute value by $2\sqrt{q_v}$. (This is done in \cite{LSV1} for the simplicial analogue  of $\PGL_d$ and $\#S_0 = 1$.)

\subsubsection{Automorphic representation and Jacquet--Langlands} Let us set 
\[
\bG :=  \GL_{1,D}
\]
for the algebraic group defined over $F$ of units of the quaternion algebra $D$. The group $\bG$ is an inner form of $\GL_2$. We will consider now the representations $\pi_v$ (resp.\ $V_\omega = \otimes_{v \in S_0} \pi_v$) as representations of $\bG(F_v) \simeq \GL_2(F_v)$ (resp.\ $\bG(F_{S_0}) = \prod_{v \in S_0} \bG(F_v) \simeq \prod_{v \in S_0} \GL_2(F_v)$, where we abbreviated $F_{S_0} = \prod_{v \in S_0} F_v$) with trivial central character.  In particular, the Hecke eigenvector $f_\omega$ is now a smooth function
\[
f_\omega : \bG(F_{S_0}) \to \bC.
\]
We denote by $\bA_F$ the ring of adels of the global field $F$. Using strong approximation for $\SL_{1,D}$ and  the fact that our representations have trivial central character, we may extend $f_\omega$ to an automorphic form (still denoted $f_\omega$)
\[
f_\omega : \bG(F) \backslash \bG(\bA_F) \to \bC.
\]
The corresponding irreducible admissible automorphic representation $\pi$ of $\bG(\bA_F)$ has local components $\pi_v$, for all $v \in S_0$, considered as above as representations with trivial central character.

The Jacquet--Langlands correspondence \cite[\S10]{gelbart} translates $\pi$ into an irreducible cuspidal automorphic representation $\rho$ of $\GL_2$ over $F$. As the local components $\pi_v$ are infinite dimensional for $v \in S_0$, the same holds for the local components $\rho_v$ of $\rho$. Moreover, since the global Jacquet--Langlands correspondence is compatible with its local counterpart \cite[\S10]{gelbart}, the eigenvalues of $A_v$ are up to sign the corresponding Hecke eigenvalues for $\rho_v$. 

The claim on the non-trivial spectrum of $A_v$ now follows from the 
Ramanujan--Petersson conjecture for $\GL_2$ over function fields, in fact a theorem of Drinfeld \cite{drinfeldRP}, see also \cite[Thm VI.10]{lafforgue}. Indeed, the components $\rho_v$ for $v \in S_0$ are spherical and $\rho$ is an infinite dimensional cuspidal automorphic representation of $\GL_2$.


\end{document}